\documentclass[a4paper,11pt,leqno]{article}

\usepackage{amssymb, amsmath}
\usepackage{amsthm}
\usepackage{mathrsfs}
\usepackage{pb-diagram}
\usepackage[all]{xy}
\usepackage{graphics}
\usepackage[dvipdfmx]{graphicx}
\usepackage{eepic}
\usepackage{epic}
\usepackage{enumerate,mathtools}
\usepackage[usenames,dvipsnames]{color}
\usepackage[hyperindex=true, 
					colorlinks=false]{hyperref}
\numberwithin{equation}{section}

\newtheorem{thm}{Theorem}[section]
\newtheorem{crl}[thm]{Corollary}

\newtheorem{lmm}[thm]{Lemma}

\newtheorem{prp}[thm]{Proposition}

\theoremstyle{definition}
\newtheorem{dfn}[thm]{Definition}
\newtheorem{exa}[thm]{Example}
\newtheorem{nota}[thm]{Notation}
\theoremstyle{remark}
\newtheorem{rem}[thm]{Remark}
\newenvironment{eq-text}
{\begin{equation} \begin{minipage}[t]{0.85\linewidth}}
{\end{minipage} \end{equation} \ignorespacesafterend}

\DeclareMathOperator{\N}{\mathbb{N}}
\DeclareMathOperator{\Z}{\mathbb{Z}}
\DeclareMathOperator{\R}{\mathbb{R}}
\DeclareMathOperator{\C}{\mathbb{C}}
\newcommand{\upa}{\uparrow}

\newcommand{\ti}{\tilde}
\newcommand{\ii}{^{-1}}
\newcommand{\pa}{\partial}
\newcommand{\ee}{\mathrm e}
\newcommand{\I}{\mathrm i}
\newcommand{\defeq}{\coloneqq} 
\newcommand{\col}{\colon\thinspace}          
\newcommand{\ens}{\enspace}
\newcommand{\ie}{{\emph{i.e.}}\ }
\newcommand{\dd}{{\mathrm d}}
\newcommand{\un}[1]{{\underline{#1}}}
\newcommand{\ov}[1]{\overline{#1}}

\newcommand{\wt}{\widetilde}

\newcommand{\dem}{\tfrac{1}{2}}
\DeclarePairedDelimiter\floor{\lfloor}{\rfloor}%
\DeclarePairedDelimiter\abs{\lvert}{\rvert}%
\newcommand{\be}{\beta}
\newcommand{\eps}{\varepsilon}
\newcommand{\ph}{\varphi}
\newcommand{\sig}{\sigma}
\newcommand{\Sig}{\Sigma}
\newcommand{\Om}{\Omega}
\newcommand{\om}{\omega}
\newcommand{\Ga}{\Gamma}
\newcommand{\ga}{\gamma}
\newcommand{\De}{\Delta}
\newcommand{\de}{\delta}
\newcommand{\la}{\lambda}

\newcommand{\tht}{\theta}
\newcommand{\ze}{\zeta}

\newcommand{\gD}{{\vec{\mathscr D}}}
\newcommand{\gE}{\mathscr E}

\newcommand{\gL}{\mathscr L}
\newcommand{\gR}{\mathscr R}
\newcommand{\gO}{\mathscr O}

\newcommand{\cB}{\mathcal{B}}
\newcommand{\cM}{\mathcal{M}}
\newcommand{\cN}{\mathcal{N}}
\newcommand{\cS}{\mathcal{S}}

\newcommand{\wrt}{with respect to}
\newcommand{\lhs}{left-hand side}
\newcommand{\rhs}{right-hand side}
\newcommand{\dfs}{d.f.s.}
\newcommand{\ddfs}{d.d.f.s.}
\newcommand{\fp}{\mathfrak{p}}
\newcommand{\fq}{\mathfrak{q}}
\newcommand{\uO}{\un{0}}
\newcommand{\uU}{\un{U}}
\newcommand{\uV}{\un{V}}
\newcommand{\uga}{\un{\ga}}
\newcommand{\uze}{\un{\ze}}
\newcommand{\vuze}{\vec{\un\ze}}

\newcommand{\Rp}{\R_{\geq0}}
\newcommand{\Zpp}{\Z_{>0}}
\newcommand{\Zmm}{\Z_{<0}}

\newcommand{\dist}{\operatorname{dist}}
\newcommand{\card}{\operatorname{card}}

\newcommand{\Det}{\operatorname{det}}

\newcommand{\begla}{\begin{equation}}
\newcommand{\beglab}[1]{\begin{equation}	\label{#1}}
\newcommand{\edla}{\end{equation}}
\newcommand{\dst}{\displaystyle}

\newcommand{\ID}{\mathop{\hbox{{\rm Id}}}\nolimits}

\newcommand{\isom}{\xrightarrow{\smash{\ensuremath{\sim}}}}

\newcommand\restr[2]{{
  \left.\kern-\nulldelimiterspace 
  #1 
  \right|_{#2} 
  }}

\newcommand{\bcb }
{\begin{color}{blue} }
\newcommand{\ecb }
{\end{color} }

\newcommand{\bcr }
{\begin{color}{red} }
\newcommand{\ecr }
{\end{color} }

\newcommand{\ff}
[1]{\textcolor{blue}{#1}}

\newcommand{\bdv}{w.r.t.\ bounded direction variation}

\newcommand{\dv}{^{\,\operatorname{dv}}}

\begin{document}
%



\thispagestyle{empty}

\begin{center}
\resizebox{\linewidth}{!}{
\textbf{\Large Iterated convolutions and endless Riemann surfaces}}
\end{center}

\bigskip

\begin{center}
{\large Shingo \textsc{Kamimoto}\footnote{Graduate School of Sciences,
Hiroshima University. 1-3-1 Kagamiyama, Higashi-Hiroshima,
Hiroshima 739-8526, Japan.} 
and David \textsc{Sauzin}\footnote{IMCCE, CNRS--Observatoire de Paris, France.}}
\end{center}

\begin{abstract}      
We discuss a version of \'Ecalle's definition of resurgence, based on
the notion of endless continuability in the Borel plane.
We relate this with the notion of $\Om$-continuability, where~$\Om$
%
is a discrete filtered set or a dicrete doubly filtered set, 
and show how to construct a universal Riemann surface~$X_\Om$
whose holomorphic functions are in one-to-one correspondence with $\Om$-continuable functions.
We then discuss the $\Om$-continuability of convolution products and
give estimates for iterated convolutions of the form
$\hat\ph_1*\cdots *\hat\ph_n$.
This allows us to handle nonlinear operations with resurgent series,
e.g. substitution into a convergent power series.
\end{abstract}

\section{Introduction}\label{sec:0}

In this article, we deal with the following version of \'Ecalle's
definition of resurgence:

%
\begin{dfn}   \label{defendlesscont}
  A convergent power series $\hat\ph\in\C\{\ze\}$ is said to be
  \emph{endlessly continuable} if, for every real $L>0$, there exists
  a finite subset~$F_L$ of~$\C$ such that the holomorphic germ at~$0$
  defined by~$\hat\ph$ can be analytically continued along every
  Lipschitz path
  $\ga \col [0,1] \to \C$
of length smaller than~$L$ 
such that $\ga(0)=0$ and $\ga\big( (0,1] \big) \subset \C\setminus F_L$.
We denote by $\hat\gR\subset\C\{\ze\}$ the space of endlessly
continuable functions.
\end{dfn}


\begin{dfn}    \label{DefResSer}
A formal series $\ti\ph(z) = \sum_{j=0}^\infty \ph_j z^{-j} \in \C[[z\ii]]$ is said to be
\emph{resurgent} if 
$\hat\ph(\ze) = \sum_{j=1}^{\infty} \ph_j \frac{\ze^{j-1}}{(j-1)!}$
is an endlessly continuable function.
\end{dfn}

In other words, the space of resurgent series is
\[
\ti\gR \defeq \cB\ii(\C\de \oplus \hat\gR) \subset \C[[z\ii]],
\]
where $\cB \col \C[[z\ii]] \to \C\de \oplus \C[[\ze]]$ is the formal Borel transform,
defined 
by $\cB\ti\ph \defeq \ph_0\de + \hat\ph(\ze)$
in the notation of Definition~\ref{DefResSer}.

We will also treat the more general case of functions which are
``endlessly continuable \bdv'': we will define a space $\hat\gR\dv$ containing $\hat\gR$
and, correspondingly, a space $\ti\gR\dv$ containing $\ti\gR$, but for
the sake of simplicity, in this introduction, we stick to the simpler
situation of Definitions~\ref{defendlesscont} and~\ref{DefResSer}.


Note that the radius of convergence of an element of~$\ti\gR$ may
be~$0$.
As for the elements of~$\hat\gR$, we will usually identify a convergent power series and the holomorphic
germ that it defines at the origin of~$\C$, as well as the holomorphic
function which is thus defined near~$0$.
Holomorphic germs with meromorphic or algebraic analytic continuation
are examples of endlessly continuable functions, but the functions
in~$\hat\gR$ can have a multiple-valued analytic continuation with a
rich set of singularities.

The \emph{convolution product} is defined as the Borel image of
multiplication and denoted by the symbol~$*$:
for $\hat\ph,\hat\psi \in \C[[\ze]]$,
$\hat\ph*\hat\psi \defeq \cB(\cB\ii\hat\ph\cdot\cB\ii\hat\psi)$,
and~$\de$ is the convolution unit (obtained from
$(\C[[\ze]],*)$ by adjunction of unit).
As is well known, for convergent power series, convolution admits the
integral representation
\beglab{eqdefconvol}
\hat\ph*\hat\psi(\ze) = \int_0^\ze \hat\ph(\xi)
\hat\psi(\ze-\xi)\,\dd\xi
\edla
for~$\ze$ in the intersection of the discs of convergence of~$\hat\ph$
and~$\hat\psi$.


Our aim is to study the analytic continuation of the convolution
product of an arbitrary number of endlessly continuable functions, 
to check its endless continuability, and also to provide
bounds, so as to be able to deal with nonlinear operations on
resurgent series.
A typical example of nonlinear operation is the substitution of one or several series without constant
term $\ti\ph_1,\ldots, \ti\ph_r$ into a power series
$F(w_1,\ldots,w_r)$, defined as
\beglab{eqdefsubstF}
F(\ti\ph_1,\ldots,\ti\ph_r) \defeq \sum_{k\in\N^r} 
c_k \, \ti\ph_1^{k_1} \cdots \ti\ph_r^{k_r}
\edla
for $F = \sum_{k\in\N^r} c_k \, w_1^{k_1} \cdots w_r^{k_r}$.
One of our main results is
%
%
\begin{thm}    \label{thmsubstgR}
Let $r\ge1$ be integer.
Then, for any convergent power series $F(w_1,\ldots,w_r) \in \C\{w_1,\ldots,w_r\}$
and for any resurgent series $\ti\ph_1,\ldots,\ti\ph_r$ without
constant term, 
$ 
F(\ti\ph_1,\ldots,\ti\ph_r) \in \ti\gR.
$ 
\end{thm}
The proof of this result requires suitable bounds for the analytic continuation of
the Borel transform of each term in the \rhs\ of~\eqref{eqdefsubstF}.
Along the way, we will study the Riemann surfaces generated by
endlessly continuable functions.
%
%
%
%
We will also prove similar results for the larger spaces~$\hat\gR\dv$ and~$\ti\gR\dv$.


\medskip

Resurgence theory was developed in the early 1980s, with \cite{E81}
and \cite{E85}, and has many mathematical applications in the study of
holomorphic dynamical systems, analytic differential equations, WKB
analysis, etc.\ (see the references e.g.\ in \cite{NLresur}).
More recently, there has been a burst of activity on the use of
resurgence in Theoretical Physics, in the context of matrix models,
string theory, quantum field theory and also quantum mechanics---see
e.g.\ 
\cite{AS}, \cite{ASV}, \cite{AU}, \cite{CDU}, \cite{CSSV},
\cite{DU12}, \cite{DU14}, \cite{GGS}, \cite{MM}.
In almost all these applications, it is an important fact that the
space of resurgence series be stable under nonlinear operations:
such stability properties are useful, and at the same time they account for the
occurrence of resurgent series in concrete problems.


These stability properties were stated in a very general framework in
\cite{E85}, but without detailed proofs,
and the part of \cite{CNP} which tackles this issue contains
obscurities and at least one mistake.
It is thus our aim in this article to provide a rigorous treatment of
this question, at least in the slightly narrower context of endless
continuability.
%
%
The definitions of resurgence that we use for~$\ti\gR$ and~$\ti\gR\dv$
are indeed more restrictive than \'Ecalle's most general definition
\cite{E85}.
%
%
In fact, our definition of~$\ti\gR\dv$ is almost identical to the one used by
Pham et al.\ in \cite{CNP}, and our definition of~$\ti\gR$
is essentially equivalent to the definition used in \cite{DO}, but the latter
preprint has flaws which induced us to develop the results of the
present paper.
These versions of the definition of resurgence
%
%
are sufficient for a
large class of applications, which virtually contains all the
aforementioned ones---see for instance \cite{ShingoK} for the details
concerning the case of nonlinear systems of differential or difference
equations.
The advantage of the definitions based on endless continuability is
that they allow for a description of the location of the singularities
in the Borel plane by means of \emph{discrete filtered sets} 
or \emph{discrete doubly filtered sets} 
(defined in Sections~\ref{secsub:dfs} and~\ref{secddfsgenres});
the notion of discrete (doubly) filtered set, adapted from \cite{CNP} and
\cite{DO}, is flexible enough to allow for a control of the
singularity structure of convolution products.


\medskip

A more restrictive definition is used \cite{NLresur} and
\cite{SLNM} (see also \cite{E81}): 
\begin{dfn}    \label{dfn:Sigcont}
Let $\Sig$ be a closed discrete subset of~$\C$. A convergent power
series~$\hat\ph$ is said to be \emph{$\Sig$-continuable} if it can be
analytically continued along any path which starts in its disc of
convergence and stays in~$\C\setminus\Sig$.
The space of $\Sig$-continuable functions is denoted by~$\hat\gR_\Sig$.
\end{dfn}

This is clearly a particular case of Definition~\ref{defendlesscont}:
any $\Sig$-continuable function is endlessly continuable
%
%
(take $F_L = \{\, \om \in \Sig \mid \abs{\om} \le L \,\}$).
It is proved in \cite{SLNM} that, if $\Sig'$ and~$\Sig''$ are closed
discrete subsets of~$\C$, and if also
$\Sig \defeq \{\om'+\om''\mid \om'\in\Sig',\; \om''\in\Sig''\}$
is closed and discrete, then 
$\hat\ph\in\hat\gR_{\Sig'},\; \hat\psi\in\hat\gR_{\Sig''}\;
\Rightarrow\; \hat\ph*\hat\psi \in \hat\gR_\Sig$.
This is because in formula~\eqref{eqdefconvol}, heuristically, singular
points tend to add to create new singularities;
so, the analytic continuation of $\hat\ph*\hat\psi$ along a path which does not stay close
to the origin is possible provided the path avoids~$\Sig$.
In particular, if a closed discrete set~$\Sig$ is closed under addition, then~$\hat\gR_\Sig$
is closed under convolution; moreover, in this case, bounds for the analytic continuation
of iterated convolutions $\hat\ph_1*\cdots*\hat\ph_n$ are given in \cite{NLresur},
where an analogue of Theorem~\ref{thmsubstgR} is proved for
$\Sig$-continuable functions.

The notion of $\Sig$-continuability is sufficient to cover interesting
applications, e.g.\ differential equations of the saddle-node
singularity type or difference equations like Abel's equation for
one-dimensional tangent-to-identity diffeomorphisms, in which
cases one may take for~$\Sig$ a one-dimensional lattice of~$\C$.
However, reflecting for a moment on the origin of resurgence in
differential equations, one sees that one cannot handle situations
beyond a certain level of complexity without replacing
$\Sig$-continuability by a more general notion like endless
continuability. Let us illustrate this point on two examples.


\begin{enumerate}[--]

\item
The equation 
$\dfrac{\dd\ph}{\dd z} - \la\ph = b(z)$,
where $b(z)$ is given in $z\ii\C\{z\ii\}$ and $\la\in\C^*$,
has a unique formal solution in $\C[[z\ii]]$, namely
$\ti\ph(z) \defeq -\la\ii\Big(\ID-\la\ii\dfrac{\dd\;}{\dd z}\Big)\ii b$,
whose Borel transform is
$\hat\ph(\ze) = -(\la+\ze)\ii \hat b(\ze)$;
here, the Borel transform~$\hat b(\ze)$ of~$b(z)$ is entire, hence
$\hat\ph$ is meromorphic in~$\C$, with at worse a pole at $\ze=-\la$ and no
singularity elsewhere.
Therefore, heuristically, for a nonlinear equation
\[
\dfrac{\dd\ph}{\dd z} - \la\ph = b_0(z) + b_1(z)\ph
+b_2(z)\ph^2 + \cdots
\]
with $b(z,w) = \sum b_m(z)w^m \in z\ii\C\{z\ii,w\}$ given,
we may expect a formal solution whose Borel transform~$\hat\ph$ has
singularities at $\ze = - n\la$, $n\in\Z_{>0}$ (because, as an effect
of the nonlinearity, the singular points tend to add),
\ie $\hat\ph$ will be $\Sig$-continuable
with $\Sig = \{-\la,-2\la,\ldots\}$ (see \cite{mouldSN} for a rigorous
proof of this),
but in the multidimensional case, for a system of~$r$ coupled
equations with \lhs s of the form 
$\dfrac{\dd\ph_j}{\dd z} - \la_j\ph_j$ with
$\la_1,\ldots,\la_r\in\C^*$,
we may expect that the Borel transforms~$\hat\ph_j$ of the components of the
formal solution have singularities at the points
$\ze = - (n_1\la_1+\cdots+n_r\la_r)$, $n\in\Z_{>0}^{\,r}$;
this set of possible singular points may fail to be closed and
discrete (depending on the arithmetical properties of
$(\la_1,\ldots,\la_r)$),
hence, in general, we cannot expect these Borel transforms to be
$\Sig$-continuable for any~$\Sig$.
Still, this does not prevent them from being always endlessly
continuable, as proved in \cite{ShingoK}.


\item
Another illustration of the need to go beyond $\Sig$-continuability
stems from parametric resurgence \cite{EcaCinq}.
Suppose that we are given a holomorphic function $b(t)$ globally
defined on~$\C$, with isolated singularities $\om\in S\subset\C$,
e.g.\ a meromorphic function, and consider the differential equation
\beglab{eqresparam}
\frac{\dd\ph}{\dd t} - z \la \ph = b(t),
\edla
where $\la\in\C^*$ is fixed and~$z$ is a large complex parameter \wrt\
which we consider perturbative expansions.
It is easy to see that there is a unique solution which is formal
in~$z$ and analytic in~$t$, namely
$\ti\ph(z,t) \defeq -\sum_{k=0}^\infty \la^{-k-1} z^{-k-1} b^{(k)}(t)$,
and its Borel transform $\hat\ph(\ze,t) = - \la\ii b(t+\la\ii\ze)$
is singular at all points of the form
$\ze_{t,\om} \defeq \la(-t+\om)$, $\om\in S$.
Now, if we add to the \rhs\ of~\eqref{eqresparam} a perturbation which
is nonlinear in~$\ph$, we can expect to get a formal solution whose
Borel transform possesses a rich set of singular points generated by
the $\ze_{t,\om}$'s, which might easily be too rich to allow for
$\Sig$-continuability with any~$\Sig$; however, we can still hope
endless continuability.

\end{enumerate}


\medskip

These are good motivations to study endless continuable functions.
As already alluded to, we will use discrete filtered sets (\dfs\ for short) to work with
them.
A \dfs\ is a family of sets $\Om=(\Om_L)_{L\in\Rp}$, where
each~$\Om_L$ is a finite set;
we will define $\Om$-continuability when~$\Om$ is a \dfs, thus
extending Definition~\ref{dfn:Sigcont}, and the space of endlessly
continuable functions will appear as the totality of $\Om$-continuable
functions for all possible \dfs\ 
This was already the approach of \cite{CNP}, and it was used in
\cite{DO} to prove that the convolution product of two
endlessly continuable functions is endlessly continuable, hence $\ti\gR$
is a subring of $\C[[z\ii]]$.
However, to reach the conclusions of Theorem~\ref{thmsubstgR}, 
we will need to give 
precise estimates on the convolution product of an arbitrary number of
endlessly continuable functions,
so as to prove the convergence of the series of holomorphic functions 
$\sum c_k \, \hat\ph_1^{*k_1} * \cdots * \hat\ph_r^{*k_r}$
(Borel transform of the \rhs\ of~\eqref{eqdefsubstF}) and to check its
endless continuability.
We will proceed similarly in the case of endless continuability \bdv,
using discrete doubly filtered sets.

Notice that explicit bounds for iterated
convolutions can be useful in themselves;
in the context of $\Sig$-continuability, such bounds were obtained in
\cite{NLresur} and they were used in \cite{KKK} in a study in WKB
analysis, where the authors track the analytic dependence upon
parameters in the exponential of the Voros coefficient.

As another contribution to the study of endlessly continuable
functions, we will show how to contruct, for each discrete filtered
set~$\Om$, a universal Riemann surface~$X_\Om$ whose holomorphic
functions are in one-to-one correspondence with $\Om$-continuable
functions.


\medskip

The plan of the paper is as follows.
\begin{enumerate}[--]
\item
Section~\ref{sec:dfs} introduces discrete filtered sets, the corresponding $\Om$-continuable
functions and their Borel images, the $\Om$-resurgent series, and
discusses their relation with Definitions~\ref{defendlesscont} and~\ref{DefResSer}.
The case of discrete doubly filtered sets and the spaces $\hat\gR\dv$
and~$\ti\gR\dv$ is in Section~\ref{secddfsgenres}.

\item
Section~\ref{sec:3} discusses the notion of $\Om$-endless Riemann
surface and shows how to construct a universal object~$X_\Om$ (Theorem~\ref{lemuniversalXOm}).

\item
In Section~\ref{sec:4}, we state and prove Theorem~\ref{thm:4.9}
which gives precise estimates for the convolution product of an
arbitrary number of endlessly continuable functions.
We also show how the analogous statement for functions which are
endlessly continuable \bdv.

\item
Section~\ref{sec:5} is devoted to applications of
Theorem~\ref{thm:4.9}:
the proof of Theorem~\ref{thmsubstgR} and even of a more general
and more precise version, Theorem~\ref{thmsubstOmgR},
and an implicit resurgent function theorem, Theorem~\ref{thm:IRFT}.

\end{enumerate}


Some of the results presented here have been announced in \cite{KS}.



\section{Discrete filtered sets and $\Om$-continuability}\label{sec:dfs}

In this section, we review the notions concerning discrete filtered
sets (usually denoted by the letter~$\Om$), the corresponding
$\Om$-allowed paths and $\Om$-continuable functions.
The relation with endless continuability is established, and sums of
discrete filtered sets are defined in order to handle convolution of
enlessly continuable functions.

\subsection{Discrete filtered sets} \label{secsub:dfs}

We first introduce the notion of discrete filtered sets which will be
used to describe singularity structure of endlessly continuable
functions
(the first part of the definition is adapted from \cite{CNP} and
\cite{DO}):

%
\begin{dfn}\label{dfn:2.1}
We use the notation
$\Rp = \{\la\in\R \mid \la\geq0\}$.
\begin{enumerate}[1)]
\item A \emph{discrete filtered set}, or \emph{\dfs}\ for short, is a
  family $\Om = (\Om_L)_{L\in\Rp}$ where
\begin{enumerate}[i)]
\item
$\Om_L$ is a finite subset of~$\C$ for each~$L$, 
\item
$\Om_{L_1}\subseteq \Om_{L_2}$ for $L_1\leq L_2$,
\item
there exists $\de>0$ such that $\Om_\de=\O$.
\end{enumerate}

\item
Let $\Om$ and $\Om'$ be \dfs\ 
We write $\Om\subset\Om'$
if $\Om_L\subset\Om'_L$ for every~$L$.

\item
We call \emph{upper closure} of a \dfs~$\Om$ the family of sets
$\ti\Om = (\ti\Om)_{L\in\Rp}$ defined by
\beglab{eq:deftiOmL}
\ti\Om_L \defeq \bigcap_{\eps>0} \Om_{L+\eps}
\quad\text{for $L\in\Rp$.}
\edla
It is easy to check that $\ti\Om$ is a \dfs\ and
$\Om \subset \ti\Om$.
\end{enumerate}
\end{dfn}

\begin{exa}   \label{exacloseddiscOm}
Given a closed discrete subset~$\Sig$ of~$\C$, the formula
\[
\Om(\Sig)_L \defeq \{\, \om \in \Sig \mid \abs{\om} \le L \,\}
\quad\text{for $L\in\Rp$}
\]
defines a \dfs\ $\Om(\Sig)$ which coincides with its upper closure.
\end{exa}

From the definition of \dfs,
we find the following
\begin{lmm}   \label{lemstructOm}
For any \dfs~$\Om$, there exists a real sequence
$(L_n)_{n\ge0}$ such that
$0=L_0 < L_1 < L_2 < \cdots$
and, for every integer $n\ge0$,
\[
L_n < L < L_{n+1} 
\quad \Rightarrow \quad
\ti\Om_{L_n} = \ti\Om_L = \Om_L.
\]
\end{lmm}


\begin{proof}
First note that~\eqref{eq:deftiOmL} entails
\beglab{eq:tiOmLtiOmLeps}
\ti\Om_L \defeq \bigcap_{\eps>0} \ti\Om_{L+\eps}
\quad\text{for every $L\in\Rp$}
\edla
(because $\Om_{L+\eps} \subset \ti\Om_{L+\eps} \subset \ti\Om_{L+2\eps}$).
Consider the weakly order-preserving integer-valued function
$L\in\Rp \mapsto \cN(L) \defeq \card\ti\Om_L$.
For each~$L$ the sequence $k\mapsto\cN(L+\frac{1}{k})$ must be
eventually constant, hence there exists $\eps_L>0$ such that,
for all $L' \in (L,L+\eps_L]$, $\cN(L') = \cN(L+\eps_L)$, whence
$\ti\Om_{L'}=\ti\Om_{L+\eps_L}$, and in fact, by~\eqref{eq:tiOmLtiOmLeps},
this holds also for $L'=L$.
The conclusion follows from the fact that
%
$\dst\Rp = \bigsqcup_{k\in\Z} \cN\ii(k)$
and each non-empty $\cN\ii(k)$ is convex, hence an interval, which
by the above must be left-closed and right-open,
hence of the form $[L,L')$ or $[L,\infty)$.
\end{proof}


Given a \dfs~$\Om$, we set
\beglab{eq:defcSOm}
\cS_\Om \defeq \big\{ (\la,\om)\in\R\times\C \mid
\la\ge0 \;\text{and}\; 
\om\in\Om_\la \big\}
\edla
and denote by~$\ov\cS_\Om$ the closure of $\cS_\Om$ in $\R\times\C$. 
We then call
\beglab{eq:defcMOm}
\cM_\Om \defeq \big(\R\times\C\big) \setminus \ov\cS_\Om
\quad \text{(open subset of $\R\times\C$)}
\edla
the \emph{allowed open set} associated with~$\Om$.
%


\begin{lmm}   \label{lemclosOm}
One has
$ 
\ov\cS_\Om = \cS_{\ti\Om}
$ and $ 
\cM_\Om = \cM_{\ti\Om}.
$ 
\end{lmm}

\begin{proof}
Suppose $(\la,\om) \in \cS_{\ti\Om}$. Then $\om\in\Om_{\la+1/k}$ for
each $k\ge1$, hence $(\la+\frac{1}{k},\om) \in \cS_\Om$, whence $(\la,\om)\in\ov\cS_\Om$.

Suppose $(\la,\om)\in\ov\cS_\Om$. Then there exists a sequence
$(\la_k,\om_k)_{k\ge1}$ in~$\cS_\Om$ which converges to $(\la,\om)$.
If $\eps>0$, then $\la_k\le \la+\eps$ for~$k$ large enough, hence
$\om_k\in\Om_{\la+\eps}$, whence $\om\in\Om_{\la+\eps}$ (because a
finite set is closed); therefore $(\la,\om) \in \cS_{\ti\Om}$.

Therefore, $\cS_{\ti\Om} = \ov\cS_\Om = \ov\cS_{\ti\Om}$ and $\cM_{\ti\Om}=\cM_\Om$.
\end{proof}

\subsection{$\Om$-allowed paths} \label{secsub:Omallpaths}

When dealing with a Lipschitz path $\ga\col [a,b] \to \C$, we denote
by~$L(\ga)$ its length.

We denote by~$\Pi$ the set of all Lipschitz paths
$\ga \col [0,t_*] \to \C$ such that $\ga(0)=0$,
with some real $t_*\ge0$ depending on~$\ga$.
Given such a $\ga\in\Pi$ and $t\in[0,t_*]$, we denote by
\[
\ga_{|t} \defeq \restr{\ga}{[0,t]} \in \Pi
\]
 the restriction of~$\ga$ to the interval $[0,t]$.
%
%
Notice that $L(\ga_{|t})$ is also Lipschitz continuous on $[0,t_*]$
since $\ga'$ exists a.e.\ and is essentially bounded by Rademacher's theorem.
%

%
\begin{dfn}   \label{defOmallowedpaths}
Given a \dfs~$\Om$,
we call \emph{$\Om$-allowed path} any $\ga\in\Pi$ such that
\[
\ti\ga(t) \defeq \big( L(\ga_{|t}), \ga(t) \big) \in \cM_\Om
\ens\text{for all $t$.}
\]
We denote by $\Pi_\Om$ the set of all $\Om$-allowed paths.
\end{dfn}


Notice that, given $t_*\ge0$,
\begin{eq-text}   \label{eqtextcharacOmalltiga}
if $t \in [0,t_*] \mapsto 
\ti\ga(t) = \big( \la(t),\ga(t) \big) \in \cM_\Om$
is a piecewise $C^1$ path such that $\ti\ga(0)=(0,0)$ and 
$\la'(t) = \abs{\ga'(t)}$ for a.e.~$t$, then $\ga\in\Pi_\Om$.
\end{eq-text}


In view of Lemmas~\ref{lemstructOm} and~\ref{lemclosOm},
we have the following characterization
of $\Om$-allowed paths:

\begin{lmm}   \label{lemOmallowedness}
Let~$\Om$ be a \dfs\ 
Then $\Pi_\Om=\Pi_{\ti\Om}$ and,
given $\ga\in\Pi$,
the followings are equivalent:
\begin{enumerate}[1)]
\item
$\ga \in \Pi_\Om$,
\item
$\ga(t) \in \C \setminus \ti\Om_{L(\ga_{|t})}$
for every~$t$,
\item
for every~$t$, there exists~$n$ such that
$L(\ga_{|t}) < L_{n+1}$
and
$\ga(t) \in \C \setminus \ti\Om_{L_n}$
\end{enumerate}
(using the notation of Lemma~\ref{lemstructOm}).
\end{lmm}

\begin{proof}
Obvious.
\end{proof}

\begin{nota}   \label{nota:cMdeL}
For $L,\de>0$, we set
\begin{gather}
\label{eq:defcMdeL}
\cM_\Om^{\de,L} \defeq \big\{\,
(\la,\ze) \in \R\times\C \mid
\dist\big( (\la,\ze), \cS_\Om \big) \ge \de
\;\text{and}\; \la\leq L
\big\}, \\[1ex]
\label{eq:defPideL}
\Pi_\Om^{\de,L} \defeq \big\{\, \ga\in\Pi_\Om \mid
\big(L(\ga_{|t}),\ga(t)\big)\in\cM_\Om^{\de,L}
\;\, \text{for all~$t$} \,\big\},
\end{gather}
where 
$\dist(\cdot\,,\cdot)$ is the Euclidean distance in 
$\R\times\C\simeq \R^3$.
\end{nota}
%
Note that 
$$
\cM_\Om=
\bigcup_{\de,L>0}
\cM_\Om^{\de,L},
\qquad
\Pi_\Om=
\bigcup_{\de,L>0}
\Pi_\Om^{\de,L}.
$$


\subsection{$\Om$-continuable functions and $\Om$-resurgent
  series}


\begin{dfn}   \label{defOmcontOmres}
Given a \dfs~$\Om$, we call \emph{$\Om$-continuable function} a
holomorphic germ $\hat\ph \in \C\{\ze\}$ which can be analytically
continued along any path $\ga\in\Pi_{\Om}$.
We denote by~$\hat\gR_\Om$ the set of all $\Om$-continuable functions
and define
\[
\ti\gR_\Om \defeq \cB\ii \big( \C\de \oplus \hat\gR_\Om \big) \subset \C[[z\ii]]
\]
to be the set of \emph{$\Om$-resurgent series}.
\end{dfn}
%

\begin{rem}
Given a closed discrete subset~$\Sig$ of~$\C$,
the $\Sig$-continuability in the sense of Definition~\ref{dfn:Sigcont}
is equivalent to the $\Om(\Sig)$-continuability
in the sense of Definition~\ref{defOmcontOmres}
for the \dfs\ $\Om(\Sig)$ of Example \ref{exacloseddiscOm}.
\end{rem}


\begin{rem}    \label{reminclusdfs}
Observe that 
$\Om\subset\Om'$ implies $\cS_\Om\subset\cS_{\Om'}$, hence
$\cM_{\Om'} \subset \cM_\Om$ and $\Pi_{\Om'} \subset \Pi_\Om$,
therefore
\[
\Om \subset \Om' \quad\Rightarrow\quad
\hat\gR_\Om \subset \hat\gR_{\Om'}.
\]
%
%
\end{rem}



\begin{rem}
  Notice that, for the trivial \dfs\ $\Om=\O$,
  $\hat\gR_{\O}=\gO(\C)$, hence $\gO(\C) \subset \hat\gR_\Om$ for
  every \dfs~$\Om$, \ie entire functions are always $\Om$-continuable.
Consequently, convergent series are always $\Om$-resurgent:
$\C\{z\ii\} \subset \ti\gR_\Om$.
However, $ \hat\gR_{\Om}=\gO(\C) $ does not imply $\Om=\O$
(consider for instance the \dfs~$\Om$ defined by $\Om_L=\O$ for
$0\le L<2$ and $\Om_L = \{1\}$ for $L\ge2$).
In fact, one can show
\[
\hat\gR_\Om = \gO(\C)
\quad \Leftrightarrow \quad
\forall L>0,\; \exists L'>L\; \text{such that} \;
\Om_{L'} \subset \{\, \om\in\C \mid \abs{\om} < L \,\}.
\]
\end{rem}


\begin{rem}   \label{rem:wlogupclos}
In view of Lemma~\ref{lemOmallowedness}, we have
$\hat\gR_{\Om}=\hat\gR_{\ti\Om}$.
Therefore, when dealing with $\Om$-resurgence, we can always suppose
that~$\Om$ coincides with its upper closure (by replacing~$\Om$ with~$\ti\Om$).
\end{rem}


We now show the relation between
resurgence in the sense of Definition~\ref{DefResSer}
and $\Om$-resurgence in the sense of Definition~\ref{defOmcontOmres}.


\begin{thm}   \label{propidgROmgR}
A formal series $\ti\ph\in\C[[z\ii]]$ is resurgent if and only if there
exists a \dfs~$\Om$ such that $\ti\ph$ is $\Om$-resurgent.
In other words,
\beglab{eqidentitygROmgR}
\hat\gR = \bigcup_{\Om\;\text{\dfs}} \hat\gR_\Om,
\qquad
\ti\gR = \bigcup_{\Om\;\text{\dfs}} \ti\gR_\Om.
\edla
\end{thm}


Before proving Theorem~\ref{propidgROmgR}, we state a technical result.


\begin{lmm}   \label{lemPathModif}
  Suppose that we are given a germ $\hat\ph \in \C\{\ze\}$ that can be
  analytically continued along a path $\ga\col[0,t_*]\to\C$ of~$\Pi$,
  and that $F$ is a finite subset of~$\C$.
Then, for each $\eps>0$, there exists a path $\ga^*\col[0,t_*]\to\C$
of~$\Pi$ such that
\begin{itemize}
\item
$\ga^*\big( (0,t_*) \big) \subset \C\setminus F$,
\item
$L(\ga^*)< L(\ga)+\eps$,
\item
$\ga^*(t_*)=\ga(t_*)$, 
the germ~$\hat\ph$ can be analytically continued along~$\ga^*$ and the
analytic continuations along~$\ga$ and~$\ga^*$ coincide.
\end{itemize}
\end{lmm}


\begin{proof}[Proof of Lemma~\ref{lemPathModif}]
  Without loss of generality, we can assume that $\ga\big([0,t_*]\big)$
  is not reduced to $\{0\}$ and that $t\mapsto L(\ga_{|t})$ is
  strictly increasing.

  The analytic continuation assumption allows us to find a finite
  subdivision $0=t_0 < \cdots < t_m = t_*$ of $[0,t_*]$ together with
  open discs $\De_0,\ldots,\De_m$ 
  %
  %
  so that, for each~$k$, 
%
%
%
  $\ga(t_k)\in \De_k$, the analytic continuation of~$\hat\ph$
  along~$\ga_{|t_k}$ extends holomorphically to~$\De_k$,
%
%
and 
$\ga\big( [t_k,t_{k+1}] \big) \subset \De_{k}$ if $k<m$.
%

For each $k\ge1$, let us pick $s_k \in (t_{k-1},t_k)$ such that 
$\ga\big( [s_k,t_k] \big) \subset \De_{k-1}\cap\De_k$;
increasing the value of~$s_k$ if necessary, we can assume
$\ga(s_k) \notin F$.
Let us also set $s_0\defeq 0$ and $s_{m+1}\defeq t_*$, so that
\[
0 \le k \le m \quad \Rightarrow \quad \left\{
\begin{aligned}
&\ga\big( [s_k,s_{k+1}] \big) \subset \De_k, \\[1ex]
&\text{the analytic continuation of~$\hat\ph$
  along~$\ga_{|s_k}$ is holomorphic in~$\De_k$} \\[1ex]
& \ga(s_k)\notin F \;\text{except maybe if $k=0$},\\[1ex]
& \ga(s_{k+1})\notin F \;\text{except maybe if $k=m$.}
\end{aligned} \right.
\]
We now define~$\ga^*$ by specifying its restriction $\restr{\ga^*}{[s_k,s_{k+1}]}$ for
each~$k$ so that it has the same endpoints as $\restr{\ga}{[s_k,s_{k+1}]}$ and,
\begin{enumerate}[--]
\item
if the open line segment $S\defeq \big(\ga(s_k),\ga(s_{k+1}) \big)$ is
contained in $\C\setminus F$, then we let
$\restr{\ga^*}{[s_k,s_{k+1}]}$ 
start at~$\ga(s_k)$ and end at~$\ga(s_{k+1})$
following~$S$, by setting
\[
\ga^*(t) \defeq \ga(s_k) + \tfrac{t-s_k}{s_{k+1}-s_k} \big( \ga(s_k)-\ga(s_{k+1}) \big)
\quad \text{for $t\in[s_k,s_{k+1}]$,}
\]
\item if not, then $S\cap F = \{\om_1,\ldots,\om_\nu\}$ with $\nu\ge1$
  (depending on~$k$); we pick $\rho>0$ small enough so that
\[
\pi\rho < \min\big\{
\dem\abs{\om_i-\ga(s_k)}, \; \dem\abs{\om_i-\ga(s_{k+1}) },\; \dem\abs{\om_j-\om_i},\;
\tfrac{\eps}{\nu(m+1)} \mid 1\le i,j,\le\nu,\; i\neq j
\big\}
\]
and we let $\restr{\ga^*}{[s_k,s_{k+1}]}$ follow~$S$ except that it
circumvents each~$\om_i$ by following a half-circle of radius~$\rho$
contained in~$\De_k$.
\end{enumerate}
This way, $\restr{\ga^*}{[s_k,s_{k+1}]}$ stays in~$\De_k$; the resulting
path $\ga^* \col [0,t_*] \to \C$ is thus a path of analytic continuation
for~$\hat\ph$ and the analytic continuations along~$\ga$ and~$\ga^*$
coincide.
On the other hand, the length of $\restr{\ga^*}{[s_k,s_{k+1}]}$ is $< \abs{\ga(s_k)-\ga(s_{k+1})} + \frac{\eps}{m+1}$,
whereas the length of $\restr{\ga}{[s_k,s_{k+1}]}$ is $\ge
\abs{\ga(s_k)-\ga(s_{k+1})}$, 
hence $L(\ga^*) < L(\ga) + \eps$.
\end{proof}


\begin{proof}[Proof of Theorem~\ref{propidgROmgR}]
%
%
Suppose first that~$\Om$ is a \dfs\ and $\hat\ph \in \hat\gR_\Om$.
Then, for every $L>0$, $\hat\ph$ meets the requirement of
Definition~\ref{defendlesscont} with $F_L = \ti\Om_L$,
hence $\hat\ph\in\hat\gR$.
Thus $\hat\gR_\Om \subset \hat\gR$, which yields one inclusion in~\eqref{eqidentitygROmgR}.
%

Suppose now $\hat\ph\in\hat\gR$. 
In view of Definition~\ref{defendlesscont}, the radius of
convergence~$\de$ of~$\hat\ph$ is positive and, for each positive
integer~$n$, we can choose a finite set~$F_{n}$ such that
\begin{eq-text}   \label{eqpropertyFnde}
the germ~$\hat\ph$ can be analytically continued along any path
$\ga\col[0,1]\to\C$ of~$\Pi$ such that $L(\ga) < (n+1)\de$ and
$\ga\big( (0,1] \big) \subset \C\setminus F_{n}$.
\end{eq-text}
Let $F_0\defeq \O$. The property~\eqref{eqpropertyFnde} holds for
$n=0$ too.
For every real $L\ge0$, we set
\[
\Om_L \defeq \bigcup_{k=0}^n F_{k}
\qquad \text{with $n \defeq \floor{L/\de}$.}
\]
One can check that $\Om\defeq (\Om_L)_{L\in\Rp}$ is a \dfs\ which coincides with its upper
closure.
We will show that $\hat\ph\in\hat\gR_\Om$.

Pick an arbitrary $\ga\col[0,1]\to\C$ such that $\ga\in\Pi_\Om$.
It is sufficient to prove that~$\hat\ph$ can be analytically
continued along~$\ga$.
Our assumption amounts to $\ga(t)\in\C\setminus\Om_{L(\ga_{|t})}$ for
each $t\in[0,1]$.
Without loss of generality, we can assume that $\ga\big([0,1]\big)$ is
not reduced to $\{0\}$ and that $t\mapsto L(\ga_{|t})$ is strictly increasing.
Let 
\[ N \defeq \floor{L(\ga)/\de}. \]
We define a subdivision $0=t_0 < t_1 < \cdots < t_N \le 1$ by the
requirement $L(\ga_{|t_n}) = n\de$ and set
\[
I_n \defeq [t_n,t_{n+1}) \ens\text{for $0 \le n < N$,}
\qquad I_N \defeq [t_N,1].
\]
For each integer~$n$ such that $0 \le n \le N$,
\beglab{eqtInLgat}
t \in I_n \quad\Rightarrow\quad n\de \le L(\ga_{|t}) < (n+1)\de,
\edla
thus $\Om_{L(\ga_{|t})} = \bigcup_{k=0}^n F_{k}$, in particular
\beglab{eqtIngatFn}
t \in I_n \quad\Rightarrow\quad \ga(t) \in \C\setminus F_{n}.
\edla
Let us check by induction on~$n$ that~$\hat\ph$ can be
analytically continued along $\ga_{|t}$ for any $t\in I_n$.

If $t\in I_0$, then $\ga_{|t}$ has length $< \de$ and the conclusion
follows from~\eqref{eqpropertyFnde}.

Suppose now that $1\le n \le N$ and that the property holds for
$n-1$. Let $t\in I_n$. 
By \eqref{eqtInLgat}--\eqref{eqtIngatFn}, we have $L(\ga_{|t}) <
(n+1)\de$ and $\ga\big( [t_n,t] \big) \subset \C\setminus F_n$.
\begin{enumerate}[--]
\item If $\ga\big( (0,t_n) \big) \cap F_n$ is empty, then the
conclusion follows from~\eqref{eqpropertyFnde}.
\item If not, then,
since $\C\setminus F_n$ is open, we can pick $t_*<t_n$ so that 
$\ga\big( [t_*,t] \big) \subset \C\setminus F_n$, 
and the induction hypothesis shows that~$\hat\ph$ can be analytically
continued along~$\ga_{|t_*}$.
We then apply Lemma~\ref{lemPathModif} to $\ga_{|t_*}$ with $F=F_n$ and
$\eps = (n+1)\de-L(\ga_{|t})$:
we get a path $\ga^* \col [0,t_*] \to \C$ which defines the same
analytic continuation for~$\hat\ph$ as~$\ga_{|t_*}$, avoids~$F_n$ and
has length $< L(\ga_{|t_*})+\eps$.
The concatenation of~$\ga^*$ with $\restr{\ga}{[t_*,t]}$ is a
path~$\ga^{**}$ of length $<(n+1)\de$ which avoids~$F_n$, so it is a
path of analytic continuation for~$\hat\ph$ because
of~\eqref{eqpropertyFnde},
and so is~$\ga$ iteself.
\end{enumerate}
\end{proof}




\subsection{Sums of discrete filtered sets}   \label{sec:sumsdsf}


It is easy to see that, if~$\Om$ and~$\Om'$ are \dfs, then the formula
\beglab{eq:defsumdfs}
(\Om *\Om')_L \defeq \{\, \om_1+\om_2 \mid
\om_1\in\Om_{L_1}, \om_2\in\Om'_{L_2}, L_1+L_2=L \, \}
\cup\Om_{L}\cup\Om'_{L}
\quad\text{for $L\in\Rp$}
\edla
defines a \dfs\ $\Om*\Om'$. We call it the \emph{sum} of~$\Om$ and~$\Om'$.

The proof of the following lemma is left to the reader.

\begin{lmm}
The law~$*$ on the set of all \dfs\ is commutative and associative.
The formula $\Om^{*n} \defeq 
\underbrace{ \Om* \cdots*\Om }_{\text{$n$ times}}$
(for $n\ge1$) defines an inductive
system, which gives rise to a \dfs
\[
\Om^{*\infty} \defeq \varinjlim_n\ \Om^{*n}.
\]
\end{lmm}


As shown in \cite{CNP} and \cite{DO}, the sum of \dfs\ is useful to study the
convolution product:


\begin{thm}[\cite{DO}]\label{thm:DO}
  Assume that $\Om$ and $\Om'$ are \dfs\ and $\hat\ph\in\hat\gR_\Om$,
  $\hat\psi\in\hat\gR_{\Om'}$.
Then the convolution product $\hat\ph * \hat\psi$ is $\Om*\Om'$-continuable.
\end{thm}


\begin{rem}   \label{rem:contrexdsf}
  Note that the notion of $\Sig$-continuability in the sense of
  Definition~\ref{dfn:Sigcont} does not give such flexibility, because
  there are closed discrete sets~$\Sig$ and~$\Sig'$ such that
  $\Om(\Sig)*\Om(\Sig') \neq \Om(\Sig'')$ for any closed
  discrete~$\Sig''$
(take e.g.\ $\Sig = \Sig' = (\Zpp\sqrt{2}) \cup \Zmm$),
and in fact there are $\Sig$-continuable functions~$\hat\ph$ such that
$\hat\ph*\hat\ph$ is not $\Sig''$-continuable for any~$\Sig''$.
\end{rem}


In view of Theorem~\ref{propidgROmgR}, a direct consequence of
Theorem~\ref{thm:DO} is that the space of endlessly continuable
functions~$\hat\gR$ is stable under convolution,
and the space of resurgent formal series~$\ti\gR$ is a subring of the
ring of formal series $\C[[z\ii]]$.


Given $\ti\ph\in\ti\gR_\Om \cap z\ii\C[[z\ii]]$, Theorem~\ref{thm:DO} guarantees the
$\Om^{*k}$-resurgence of $\ti\ph^k$ for every integer~$k$, hence its
$\Om^{*\infty}$-resurgence. 
This is a first step towards the proof of the resurgence of
$F(\ti\ph)$ for $F(w) = \sum c_k w^k \in\C\{w\}$, \ie
Theorem~\ref{thmsubstgR} in the case $r=1$,
however some analysis is needed to prove the convergence of $\sum c_k
\ti\ph^k$ in some appropriate topology.
What we need is a precise estimate for the convolution product of
an arbitrary number of endlessly continuable functions, 
and this will be the content of Theorem~\ref{thm:4.9}.
In Section \ref{sec:5}, the substitution problem will be discussed in
a more general setting, resulting in Theorem~\ref{thmsubstOmgR}, which
is more general and more precise than Theorem~\ref{thmsubstgR}.



\subsection{Discrete doubly filtered sets and a more general
  definition of resurgence}   \label{secddfsgenres}


We now define the spaces~$\hat\gR\dv$
and~$\ti\gR\dv$ which were alluded to in the introduction.
We first require the notion of ``direction variation'' of a
$C^{1+\text{Lip}}$ path.

We denote by $\Pi\dv$ the set of all $C^1$ paths~$\ga$ belonging
to~$\Pi$, such that~$\ga'$ is Lipschitz and never vanishes.
By Rademacher's theorem, $\ga''$ exists a.e.\ on the interval of
definition $[0,t_*]$ of~$\ga$ and is essentially bounded. We can thus define the \emph{direction variation}
$V(\ga)$ of $\ga\in\Pi\dv$ by
\[
V(\ga)\defeq
\int_0^{t_*}
\abs*{{\rm Im}\frac{\ga''(t)}{\ga'(t)}}
\dd t
\]
(notice that one can write $\ga'(t) = \abs{\ga'(t)}\,\ee^{\I\tht(t)}$
with a real-valued Lipschitz function~$\tht$, and then ${\rm
  Im}\frac{\ga''(t)}{\ga'(t)}=\tht'$, hence~$V(\ga)$ is nothing but the
length of the path~$\tht$).
Note that the function $t \mapsto V(\ga_{|t})$ is Lipschitz.
%


\begin{dfn}
A convergent power series $\hat\ph\in\C\{\ze\}$ is said to be
\emph{endlessly continuable \bdv} (and we write
$\hat\ph\in\hat\gR\dv$) if, for every real $L,M>0$, there exists a
finite subset~$F_{L,M}$ of~$\C$ such that~$\hat\ph$ can be analytically continued along every path
$\ga \col [0,1] \to \C$ such that $\ga\in\Pi\dv$,
$L(\ga)<L$, $V(\ga)<M$,
and $\ga\big( (0,1] \big) \subset \C\setminus F_{L,M}$.
\end{dfn}

We also set
$\ti\gR\dv \defeq \cB\ii(\C\de \oplus \hat\gR\dv)$.
Note that $\hat\gR\subset\hat\gR\dv \subset \C\{\ze\}$ and 
$\ti\gR\subset\ti\gR\dv \subset \C[[z\ii]]$.


\begin{dfn}
A \emph{discrete doubly filtered set}, or \emph{\ddfs}\ for short, is a
family $\Om = (\Om_{L,M})_{L,M\in\Rp}$ that satisfies
\begin{enumerate}[i)]
\item
$\Om_{L,M}$ is a finite subset of~$\C$ for each~$L$ and~$M$, 
\item
$\Om_{L_1,M_1}\subseteq \Om_{L_2,M_2}$ 
when $L_1\leq L_2$ and $M_1\leq M_2$,
\item
there exists $\de>0$ such that $\Om_{\de,M}=\O$
for all $M\geq0$.
\end{enumerate}
\end{dfn}
Notice that 
a \dfs\ $\Om$ can be regarded as a \ddfs\ $\Om\dv$ by
setting $\Om_{L,M}\dv:=\Om_{L}$ for $L,M\geq0$.


For a \ddfs\ $\Om$, we set
$
\cS_\Om \defeq \big\{ (\mu,\la,\om)\in\R^2\times\C \mid
\mu\ge0,\la\ge0 \;\text{and}\; 
\om\in\Om_{\la,\mu} \big\}
$
and $\cM_\Om \defeq \big(\R^2\times\C\big) \setminus \ov\cS_\Om$,
where $\ov\cS_\Om$ is the closure of $\cS_\Om$ 
in $\R^2\times\C$. 
We call \emph{$\Om$-allowed path} any $\ga\in\Pi\dv$ such that
\beglab{eqdeftigadv}
\ti\ga\dv(t) \defeq 
\big( V(\ga_{|t}),L(\ga_{|t}), 
\ga(t) \big) \in \cM_\Om
\ens\text{for all $t$.}
\edla
We denote by $\Pi_\Om\dv$ the set of all $\Om$-allowed paths.
Finally, the set of \emph{$\Om$-continuable functions}
(resp.\ \emph{$\Om$-resurgent series}) is defined in the same way as in
Definition~\ref{defOmcontOmres},
and denoted by $\hat\gR_\Om\dv$
(resp.\ $\ti\gR_\Om\dv$).
%
%
Arguing as for Theorem~\ref{propidgROmgR}, one obtains
\begin{equation}
\hat\gR\dv = \bigcup\limits_{\Om\;\text{\ddfs}} \hat\gR\dv_\Om,
\qquad
\ti\gR\dv = \bigcup\limits_{\Om\;\text{\ddfs}} \ti\gR_\Om\dv.
\end{equation}
%


The \emph{sum} $\Om*\Om'$ of two \ddfs~$\Om$ and~$\Om'$ is the \ddfs\ defined by setting,
for $L,M\in\Rp$,
\beglab{eq:defsumddfs}
(\Om *\Om')_{L,M} \defeq \{\, \om_1+\om_2 \mid
\om_1\in\Om_{L_1,M}, \om_2\in\Om'_{L_2,M}, L_1+L_2=L \, \}
\cup\Om_{L,M}\cup\Om'_{L,M}.
\edla


\section{The endless Riemann surface associated with a \dfs} 
\label{sec:3}

We introduce the notion of $\Om$-endless Riemann surfaces
for a \dfs\ $\Om$ as follows:
\begin{dfn}
We call \emph{$\Om$-endless Riemann surface} any triple $(X,\fp,\uO)$ such
that
$X$ is a connected Riemann surface,
$\fp \col X \to \C$ is a local biholomorphism,
$\uO \in \fp\ii(0)$,
and any path $\ga \col [0,1] \to \C$ of~$\Pi_\Om$ has a lift
$\uga \col [0,1] \to X$ such that $\uga(0) = \uO$.
A \emph{morphism of $\Om$-endless Riemann surfaces}
is a local biholomorphism
$
\fq:(X,\fp,\uO)\to(X',\fp',\uO')
$
that makes the following diagram commutative:
\[
\begin{xy}
(0,20) *{(X,\uO)},
(30,20) *{\ (X',\uO')},
(15,0) *{(\C,0)},
{(6,20) \ar (24,20)},
{(2,16) \ar (13,4)},
{(28,16) \ar (17,4)},
(14,23) *{\fq},
(3,10) *{\fp},
(27,10) *{\fp'}
\end{xy}
\]
\end{dfn}

In this section,
we prove the existence of an initial object
$(X_{\Om},\fp_\Om,\uO_\Om)$ in the category
of $\Om$-endless Riemann surfaces:
\begin{thm}   \label{lemuniversalXOm}
There exists an $\Om$-endless Riemann surface $(X_{\Om},\fp_\Om,\uO_\Om)$ such
that, for any $\Om$-endless Riemann surface $(X,\fp,\uO)$,
there is a unique morphism 
$$
\fq \col (X_{\Om},\fp_\Om,\uO_\Om) \to  (X,\fp,\uO).
$$
The $\Om$-endless Riemann surface $(X_{\Om},\fp_\Om,\uO_\Om)$ is unique up to
isomorphism
and $X_\Om$ is simply connected.
\end{thm}

\subsection{Construction of $X_\Om$}\label{sec:3.1}

We first define ``skeleton" 
of $\Om$:
\begin{dfn}\label{dfn:3.3}
Let
$
V_\Om\subset
\bigcup_{n=1}^{\infty}
(\C\times\Z)^n
$
be the set of vertices
$$
v\defeq 
((\om_1,\sigma_1),\cdots,(\om_n,\sigma_n))
\in
(\C\times\Z)^n
$$
that satisfy the following conditions:
\begin{enumerate}
\item[1)]
$(\om_1,\sigma_1)=(0,0)$
and
$(\om_j,\sigma_j)\in\C\times(\Z\setminus\{0\})$
for $j\geq2$,

\item[2)]
$\om_j\neq\om_{j+1}$
for $j=1,\cdots,n-1$,

\item[3)]
$\om_j \in \ti\Om_{L_j(v)}$ 
with
$
L_{j}(v)\defeq \sum_{i=1}^{j-1}
|\om_{i+1}-\om_i|
$
for $j=2,\cdots,n$.

\end{enumerate}
Let $E_\Om\subset V_\Om\times V_\Om$ be the set of edges
$e=(v',v)$ that satisfy one of the following conditions: 
\begin{enumerate}
\item[i)]
$v=((\om_1,\sigma_1),\cdots,(\om_{n},\sigma_{n}))$
and
$v'=((\om_1,\sigma_1),\cdots,(\om_{n},\sigma_{n}),
(\om_{n+1},\pm 1))$,

\item[ii)]
$v=((\om_1,\sigma_1),\cdots,(\om_{n},\sigma_{n}))$
and
$v'=((\om_1,\sigma_1),\cdots,(\om_{n},\sigma_{n}+1))$ 
with $\sigma_n\geq1$,

\item[iii)]
$v=((\om_1,\sigma_1),\cdots,(\om_{n},\sigma_{n}))$
and
$v'=((\om_1,\sigma_1),\cdots,(\om_{n},\sigma_{n}-1))$ 
with $\sigma_n\leq-1$.

\end{enumerate}
We denote
the directed tree diagram $(V_\Om,E_\Om)$
by $Sk_\Om$
and call it
\emph{skeleton} of $\Om$.
\end{dfn}

\begin{nota}
For
$
v\in
V_\Om
\cap
(\C\times\Z)^n,
$
we set
$\om(v)\defeq \om_n$
and $L(v)\defeq L_n(v)$.
\end{nota}

From the definition of $Sk_\Om$,
we find the following
\begin{lmm}
For each $v\in V_\Om\setminus\{(0,0)\}$,
there exists a unique vertex $v_\upa\in V_\Om$
such that $(v,v_\upa)\in E_\Om$.
\end{lmm}

To each $v\in V_\Om$ we assign a cut plane~$U_v$, defined as the open set 
\[
U_v
\defeq 
\C\setminus
\Big(C_v\cup
\bigcup_{v'\underset{\rm i}{\to} v}C_{v'\to v}
\Big),
\]
where $\displaystyle
\bigcup_{v'\underset{\rm i}{\to} v}
$
is the union over all the vertices $v'\in V_\Om$ that have
an edge $(v',v)\in E_\Om$ of type i),
\vspace{-3ex}
\begin{align*}
&C_v \defeq \left\{ \begin{aligned}
&\O &\text{when $v=(0,0),$}
\\[.7ex]
&\{\om_n - s(\om_n-\om_{n-1}) \mid s\in\R_{\geq0}\}
&\text{when $v\neq(0,0),$}
\end{aligned} \right.\\[1.3ex]
&C_{v'\to v} \defeq 
\{ \om_{n+1}+s(\om_{n+1}-\om_n) \mid s\in\R_{\geq0}\}.
\end{align*}
We patch the~$U_v$'s along the cuts according to the following rules:

\begin{figure}
\centering
\includegraphics[scale=0.4]{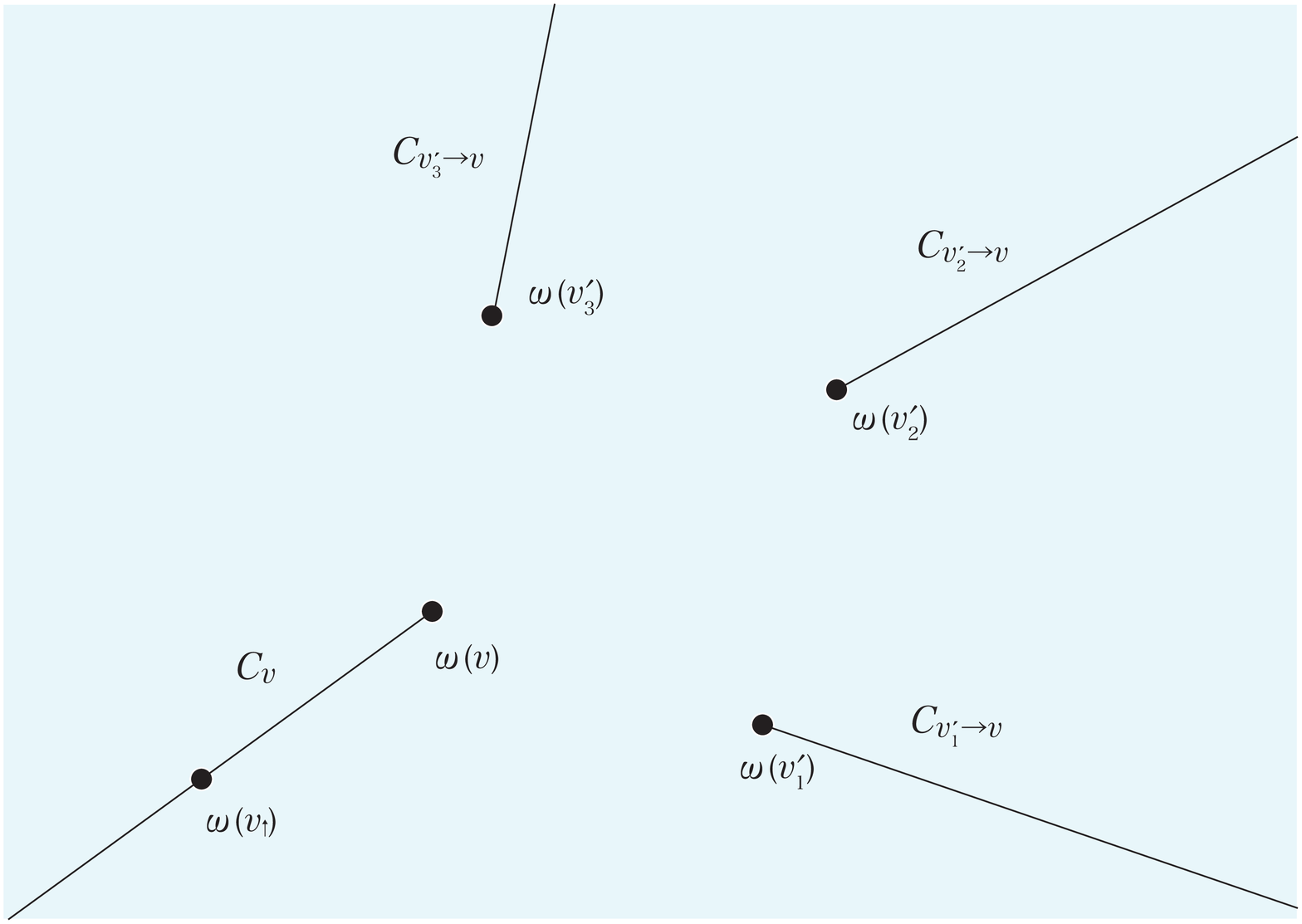}
\caption{The set $U_v$.}
\end{figure}

Suppose first that $(v',v)$ is an edge of type i), with $v'=(v,(\om_{n+1},\sigma_{n+1}))
\in V_\Om$.
To it, we assign a line segment or a half-line~$\ell_{v'\to v}$ as follows:
If there exists $u=(v,(\om'_{n+1},\pm 1))\in V_\Om$ such that
$\om'_{n+1}\in C_{v'\to v}\setminus\{\om_{n+1}\}$, take
$u^{(0)}=(v,(\om^{(0)}_{n+1},\pm 1))\in V_\Om$ so that
$|\om^{(0)}_{n+1}-\om_{n+1}|$ gives the minimum of
$|\om'_{n+1}-\om_{n+1}|$ for such vertices and assign an open line segment
$ \ell_{v'\to v} \defeq \{ \om_{n+1}+s(\om^{(0)}_{n+1}-\om_{n+1}) \ |\
s\in(0,1) \} $
to $(v',v)$.  
Otherwise, we assign the open half-line
$ \ell_{v'\to v} \defeq C_{v'\to v}\setminus\{\om_{n+1}\} $ to $(v',v)$.
Since each $\Om_L$ $(L\geq0)$ is finite, we can take a connected
neighborhood $ U_{v'\to v} $ of $ \ell_{v'\to v} $ so that
\begin{equation}\label{3.1}
U_{v'\to v}\setminus\ell_{v'\to v} =
U^+_{v'\to v}\cup U^-_{v'\to v}
\quad
\text{and}
\quad
U^{\pm}_{v'\to v}\subset U_v\cap U_{v'},
\end{equation}
where
$$
U^{\pm}_{v'\to v} \defeq 
\{\ze\in U_{v'\to v} \mid
\pm{\rm Im}(\ze\cdot\overline{\ze'})>0
\ \text{for}\ \ze'\in\ell_{v'\to v}
\}.
$$
Then, if $\sig_{n+1}=1$, we glue $U_v$ and $U_{v'}$ along $U^-_{v'\to v}$,
whereas if $\sig_{n+1}=-1$ we glue them along $U^+_{v'\to v}$.


%
%
\begin{figure}
\centering
\includegraphics[scale=0.4]{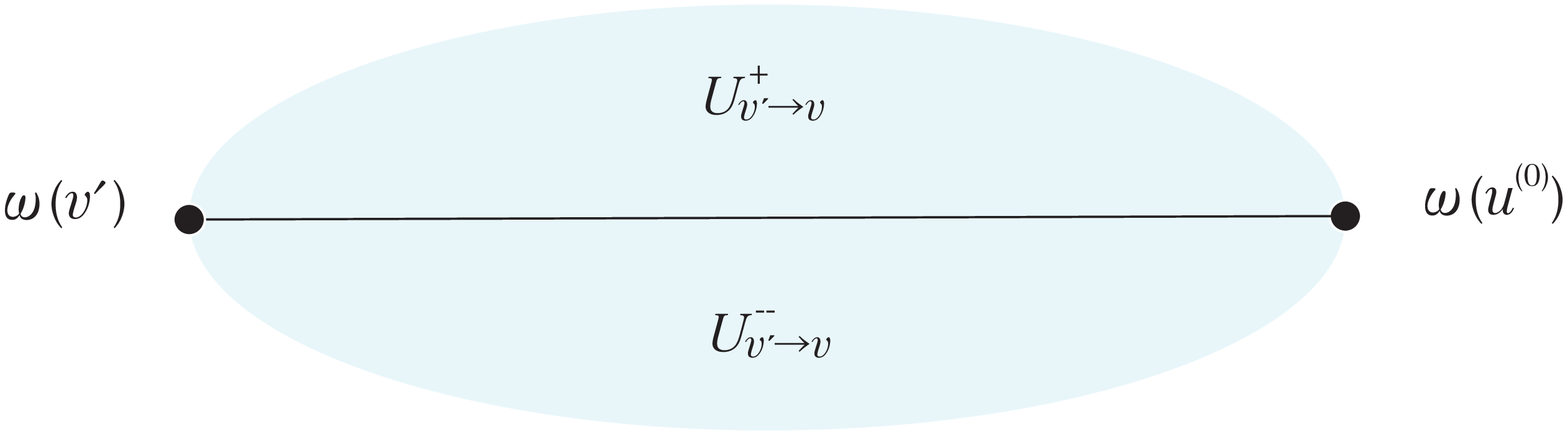}
\caption{The set $U_{v'\to v}$.}
\end{figure}
Suppose now that $(v',v)$ is an edge ot type ii) and iii).
As in the case of i), if there exists
$u=(v,(\om'_{n+1},\pm 1))\in V_\Om$ such that
$\om'_{n+1}\in C_{v}\setminus\{\om_{n}\}$,
then we take $u^{(0)}=(v,(\om^{(0)}_{n+1},\pm 1))\in V_\Om$
so that $|\om^{(0)}_{n+1}-\om_{n}|$ is minimum
and assign
$
\ell_{v'\to v} \defeq
\{
\om_{n}+s(\om^{(0)}_{n+1}-\om_{n})
\ |\ 
s\in(0,1)
\}
$
to $(v',v)$.
Otherwise, we assign
$
\ell_{v'\to v} \defeq C_v\setminus\{\om_{n}\}
$
to $(v',v)$.
Then, we take a connected neighborhood 
$U_{v'\to v}$ of $\ell_{v'\to v}$
satisfying \eqref{3.1},
and glue~$U_v$ and~$U_{v'}$ along $U^-_{v'\to v}$ in case~ii), 
and along $U^+_{v'\to v}$ in case~iii).


%
%
Patching the $U_v$'s and the $U_{v'\to v}$'s according to the above rules,
we obtain a Riemann surface~$X_\Om$, in which we denote by $\un0_\Om$
the point corresponding to $0\in U_{(0,0)}$.
The map $\fp_\Om \col X_\Om\to\C$
is naturally defined using
local coordinates
$U_v$ and $U_{v'\to v}$.

Let $\un U{}_e$, $\un\ell{}_e$ $(e\in E_\Om)$
and
$\un U{}_v$ $(v\in V_\Om)$
respectively denote
the subsets of $X_\Om$
defined by 
$U_e$, $\ell_e$
and $U_v$.
Notice that each $\un\ze\in X_\Om$
belongs to one of the $\un\ell{}_e$'s
or $\un U{}_v$'s ($e\in E_\Om$ or $v\in V_\Om$).
Therefore, we have the following decomposition of $X_\Om$:
$$
X_\Om=
\bigsqcup_{v\in V_\Om}
\un U{}_v
\sqcup
\bigsqcup_{e\in E_\Om}
\un \ell{}_e.
$$
\begin{dfn}
We define a function
$L:X_\Om\to \Rp$
by the following formula:
$$
L(\un\ze)\defeq 
L(v)+|\fp(\un\ze)-\om(v)|
\quad
\text{when}
\quad
\un\ze\in \un U{}_v \sqcup \un \ell{}_{v\to v_\upa}.
$$
We call $L(\un\ze)$ \emph{the canonical distance}
of $\un\ze$ from $\un0_\Om$.
\end{dfn}
We obtain from the construction of $L$
the following
\begin{lmm}
The function $L:X_\Om\to \Rp$
is continuous and satisfies
the following inequality
for every $\ga\in\Pi_\Om :$
$$
L(\un\ga(t))
\leq
L(\ga_{|t})
\quad
\text{for}
\quad
t\in[0,1].
$$
\end{lmm}

We now show the fundamental properties of $X_\Om$.
\begin{lmm}
The Riemann surface $X_\Om$ constructed above
is simply connected.
\end{lmm}
\begin{proof}
We first note that,
since $Sk_\Om$ is connected,
$X_\Om$ is path-connected.
Let $\un\ga:[0,1]\to X_\Om$ be a path such that
$\un\ga(0)=\un\ga(1)$.
Since the image of $\un\ga$ is compact set in $X_\Om$,
we can take finite number of vertices 
$\{v_j\}_{j=1}^{p}\subset V_\Om$
and
$\{e_j\}_{j=1}^q\subset E_\Om$
so that $v_1=(0,0)$ and
the image of $\un\ga$
is covered by $\{\un U{}_{v_j}\}_{j=1}^p$ and 
$\{\un U{}_{e_j}\}_{j=1}^q$.
Since each of 
$\{v_j\}_{j=2}^{p}$
and
$\{e_j\}_{j=1}^q$
has a path to $v_1$
that contains it,
interpolating finite number of
the vertices and the edges if necessary,
we may assume that the diagram $\wt{Sk}$
defined by
$\{v_j\}_{j=1}^{p}$
and
$\{e_j\}_{j=1}^q$
are connected in $Sk_\Om$.
Now, let $\un U$ be the union of
$\{\un U{}_{v_j}\}_{j=1}^p$ and 
$\{\un U{}_{e_j}\}_{j=1}^q$.
Since all of the open sets
are simply connected
and $\wt{Sk}$ is acyclic,
we can inductively confirm
using the van Kampen's theorem
that $\un U$ is simply connected.
Therefore, the path
$\un\ga$ is contracted to the point $\un0_\Om$.
It proves the simply connectedness of $X_\Om$.
\end{proof}
\begin{lmm}
The Riemann surface $X_\Om$ constructed above
is $\Om$-endless.
\end{lmm}
\begin{proof}
Take an arbitrary $\Om$-allowed path $\ga$
and $\de,L>0$
so that $\ga\in\Pi_\Om^{\de,L}$.
Let $V_\Om^{\de,L}$ denote the set of vertices
$v=((\om_1,\sigma_1),\cdots,(\om_{n},\sigma_{n}))\in V_\Om$ 
that satisfy
$$
L^\de(v)\defeq 
L_n(v)+
\sum_{j=2}^{n}(|\sigma_j|-1)\de
\leq L
$$
and set 
$E_\Om^{\de,L}
\defeq \{(v,v_\upa)\in E_\Om \mid
v\in V_\Om^{\de,L}\}$. 
Notice that $V_\Om^{\de,L}$ and $E_\Om^{\de,L}$
are finite.
We set
for $\eps>0$ and
$v\in V_\Om^{\de,L}$
$$
U_v^{\de,L,\eps}\defeq 
\{\ze\in U_v
\mid
\inf_{(v',v)\in E_\Om}|\ze-\om(v')|\geq\de,\ 
D_\ze^\eps\subset U_v
\}
\cap
D_{\om(v)}^{L-L^\de(v)},
$$
where
$
D_\ze^r\defeq \{\ti\ze\in\C\mid|\ti\ze-\ze|\leq r\}
$
for
$\ze\in\C, r>0$.
We also set
for $\eps>0$
and $(v,v_\upa)\in E_\Om^{\de,L}$
$$
U_{v\to v_\upa}^{\de,L,\eps}\defeq 
\{\ze\in U_{v\to v_\upa}
\mid
\min_{j=1,2}|\ze-\ti\om_j|\geq \de,\ 
\inf_{\ti\ze\in \ell_{v\to v_\upa}}|\ze-\ti\ze|\leq \eps
\}
\cap
D_{\om(v)}^{L-L^\de(v_\upa)},
$$
where
$\ti\om_1\defeq \om(v)$
and $\ti\om_2$ is the other endpoint of $\ell_{v\to v_\upa}$
if it exists and $\ti\om_2\defeq \om(v)$ otherwise.
Since
$E_\Om^{\de,L}$ are finite set,
we can take $\eps>0$ sufficiently small
so that
$
D_\ze^\eps
\subset
U_{v\to v_\upa}
$
for all
$
\ze\in
U_{v\to v_\upa}^{\de,L,\eps}
$
and
$(v,v_\upa)\in E_\Om^{\de,L}$.
We fix such a number
$\eps>0$.

Now,
let $I$ be the maximal interval such that
the restriction of
$\ga$ to $I$
has a lift $\un\ga$ on $X_\Om$.
Obviously,
$I\neq \O$ and $I$ is open.
Assume that
$I=[0,a)$ for $a\in(0,1]$.
We take $b\in (0,a)$ so that
$L(\ga_{|a})-L(\ga_{|b})< \eps$.
Then, notice that,
since $\ga\in\Pi_\Om^{\de,L}$ and
$\ga_{|b}$ has a lift on $X_\Om$,
$\un\ga(b)$ is in $\un U{}_v^{\de,L,\eps}$ for
$v\in V_\Om^{\de,L}$ or
$\un U{}_e^{\de,L,\eps}$ for
$e\in E_\Om^{\de,L}$.
Since
$D_{\ga(b)}^\eps\subset U_v$
(resp., $D_{\ga(b)}^\eps\subset U_e$)
when 
$\un\ga(b)\in \un U{}_v^{\de,L,\eps}$
(resp., $\un\ga(b)\in \un U{}_e^{\de,L,\eps}$),
we obtain a lift of $\ga|_{[0,a]}$
by concatenating $\un\ga{}_{|b}$ and
$\ga|_{[b,a]}$ in the coordinate.
It contradicts the maximality of $I$,
and hence, $I=[0,1]$.
\end{proof}

\subsection{Proof of Theorem \ref{lemuniversalXOm}}
We first show the following:
\begin{lmm}\label{lmm:3.8}
For all $\eps>0$ and
$\un\ze\in X_\Om$,
there exists an $\Om$-allowed path 
$
\ga
$ 
such that $L(\ga)<L(\un\ze)+\eps$ and
its lift $\uga$ on $X_\Om$ satisfies
$\uga(0)=\un0_\Om$ and $\uga(1)=\un\ze$.
\end{lmm}
\begin{proof}
Let $\un\ze\in \un U{}_v$ for 
$
v=((\om_1,\sigma_1),\cdots,(\om_n,\sigma_n)).
$
We consider a polygonal curve
$P_{\un\ze}^0$
obtained by connecting line segments
$[\om_j,\om_{j+1}]$ $(j=1,\cdots,n)$,
where we set
$\om_{n+1}\defeq \fp_\Om(\un\ze)$
for the sake of notational simplicity.
Now, collect all the points $\om_{j,k}$ on $(\om_j,\om_{j+1})$
such that
$(L_{j,k},\om)\in\ov\cS_{\Om}$,
where $L_{j,k}\defeq L_j(v)+|\om_{j,k}-\om_j|$.
Since 
\begin{eq-text}\label{finline}
$
\ov\cS_{\Om}
\cap\{\la\in\Rp\mid 
|\la|\leq L\}\times\C
$
is written for each $L>0$
by the union of
finite number of line segments
of the form
$
\{\la\in\Rp\mid
\ti L\leq\la\leq L\}
\times\{\om\}
$
$(\ti L>0, \om\in\C)$,
\end{eq-text}
such points are finite.
We order $\om_j$ and $\om_{j,k}$
so that $L_j(v)$ and $L_{j,k}$ increase 
along the order and
denote the sequence by
$
(\om'_1,\om'_2,\cdots,\om'_{n'}).
$
We set 
$
L'_{j}\defeq \sum_{i=1}^{j-1}
|\om'_{i+1}-\om'_i|.
$
We extend $v$ to
$
v'=((\om'_1,\sigma'_1),\cdots,(\om'_{n'},\sigma'_{n'}))
$
by setting $\sigma'_j=1$ (resp., $\sigma'_j=-1$)
when
$(\om'_j,L'_j)=(\om_{i,k},L_{i,k})$
for some $i,k$ and $\sigma_{i+1}\geq 1$
(resp., $\sigma_{i+1}\leq -1$).
Then, 
in view of \eqref{finline},
we can take $\de>0$ so that
$$
\{
(L'_j+|\ze'-\om'_j|+\de,\ze')
\ |\ 
\ze'\in(\om'_j,\om'_{j+1})
\}
\cap\ov\cS_{\Om}
=\O,
$$
$$
\{
(L'_j+\de,\ze')
\ |\ 
0<|\ze'-\om'_j|<\de
\}
\cap\ov\cS_{\Om}
=\O
$$
hold for
$j=1,\cdots,n'$.
Let $\om'_{j,-}$ (resp., $\om'_{j,+}$)
be the intersection point of
$[\om'_{j-1},\om'_{j}]$ 
(resp., $[\om'_{j},\om'_{j+1}]$)
and
$
C_{\om'_j}^{\varepsilon'}
\defeq 
\{\ze'\in\C
\ |\ 
|\ze'-\om'_j|=\varepsilon'\}
$
for sufficiently small $\varepsilon'>0$.
We replace the part 
$[\om'_{j,-},\om'_{j}]\cup[\om'_{j},\om'_{j,+}]$ of
$\ell$ with a path
that goes 
anti-clockwise (resp., clockwise)
along 
$
C_{\om'_j}^{\varepsilon'}
$
from $\om'_{j,-}$ to $\om'_{j,+}$
and turns around $\om'_j$
($|\sigma'_j|-1$)-times
when $\sigma'_j\geq1$
(resp., when $\sigma'_j\leq-1$).
Let 
$P_{\un\ze}^{\eps'}$ 
denote a path obtained from 
$P_{\un\ze}^0$
by the modification.
Then, $P_{\un\ze}^{\eps'}$ 
defines an $\Om$-allowed path and
its lift $\un P_{\un\ze}^{\eps'}$ 
on $X_\Om$ satisfies the conditions.
Further,
by taking $\eps'$ sufficiently small
so that 
$
2\pi\eps'
\sum_{j=2}^{n'}|\sigma'_j|
<\eps,
$
we find
$L(P_{\un\ze}^{\eps'})<L(\un\ze)+\eps$, hence one can take $\ga = P_{\un\ze}^{\eps'}$.
\begin{figure}
\centering
\includegraphics[scale=0.4]{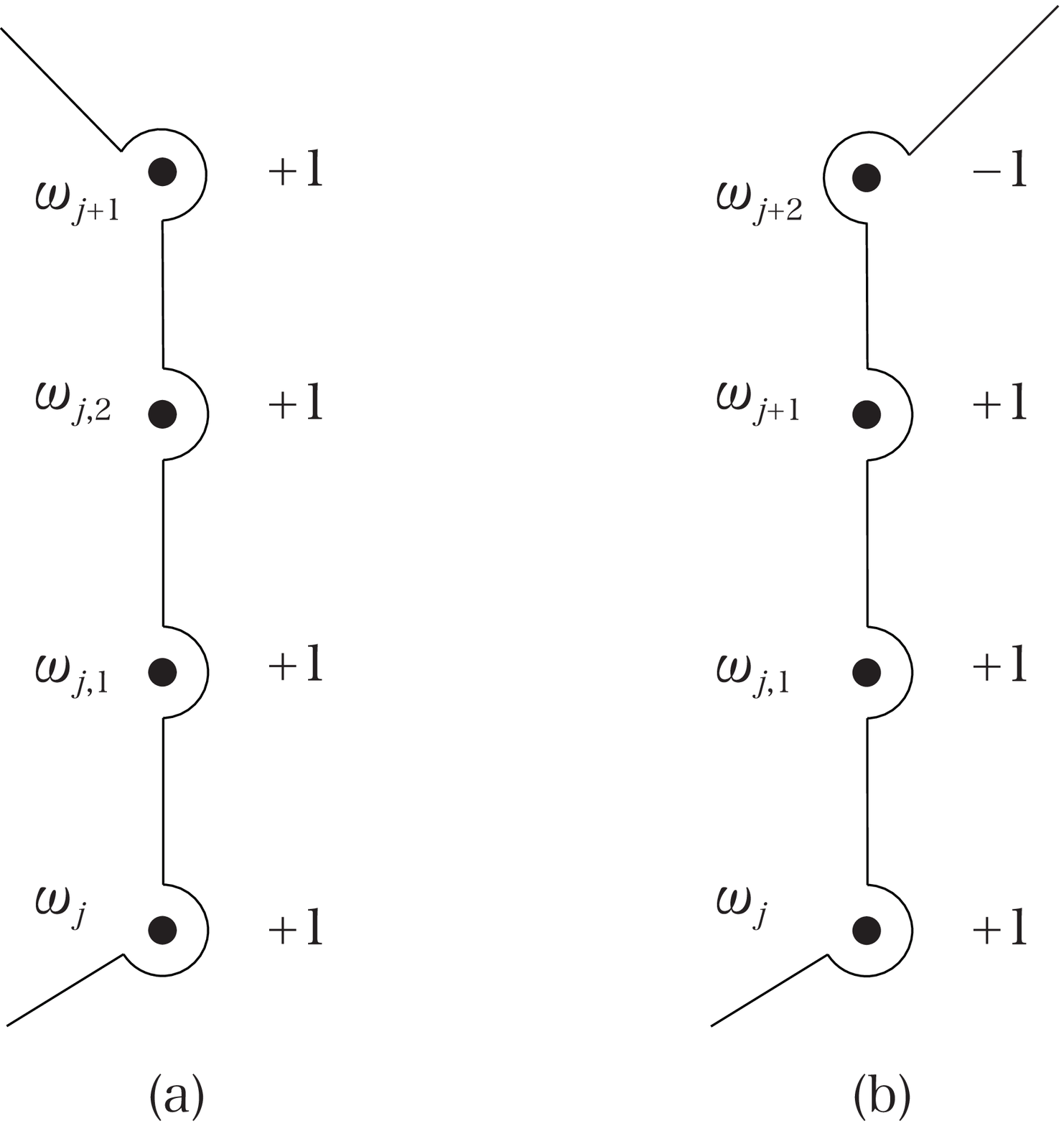}
\caption{}
\end{figure}
When $\un\ze\in \un\ell{}_{e}$ 
for an edge $e=(v,v_\upa)\in E_\Om$,
we can construct such a path 
$P_{\un\ze}^{\eps'}\in\Pi_\Om$
by totally the same discussion.
\end{proof}

Notice that,
since the sequence $v'$
in the proof of Lemma \ref{lmm:3.8}
is uniquely determined by $\un\ze\in X_\Om$,
the choice of the path
$P_{\un\ze}^{\eps'}$
depends only on the radius $\eps'$
of the circles
$
C_{\om'_j}^{\eps'}.
$
Further,
from the construction of the path $P_{\un\ze}^{\eps'}$,
we can extend Lemma \ref{lmm:3.8} as follows:
\begin{lmm}\label{lmm:3.9}
For all $\eps>0$ and
$\un\ze\in X_\Om$,
there exist 
a neighborhood $\un U_{\un\ze}$ of $\un\ze$
and, for~$\eps'$ small enough, a continuous deformation 
$Q^{\eps'}_{\un\ze,\un\ze'}\in\Pi_\Om$ 
$(\un\ze'\in\un U{}_{\un\ze})$
of the path $\ga = P^{\eps'}_{\un\ze}$ constructed in the proof of Lemma~\ref{lmm:3.8}
such that
$L(Q^{\eps'}_{\un\ze,\un\ze'})<L(\un\ze')+2\eps$
for each $\un\ze'\in\un U{}_{\un\ze}$ and
the lift $\un Q^{\eps'}_{\un\ze,\un\ze'}$ on $X_\Om$
satisfies $\un Q^{\eps'}_{\un\ze,\un\ze'}(0)=\un0$ and
$\un Q^{\eps'}_{\un\ze,\un\ze'}(1)=\un\ze'$.
\end{lmm}
Indeed,
the deformation of $P^{\eps'}_{\un\ze}$
is concretely given as follows:
\begin{enumerate}[--]
\item
When $\un\ze\in\un U_v$ for $v\in V_\Om$,
taking a neighborhood 
$\un U{}_{\un\ze}\subset\un U_v$ of $\un\ze$
sufficiently small,
we find that the family of
the paths $P^{\eps'}_{\un\ze'}$
$(\un\ze'\in\un U{}_{\un\ze})$
constructed in the proof
of Lemma \ref{lmm:3.8}
gives such a deformation.

\item
When $\un\ze\in\un \ell_e$ for $e\in E_\Om$,
we can take a neighborhood 
$\un U{}_{\un\ze}\subset\un U_e$ of $\un\ze$
so that
$
[\om'_{n',+}(\un\ze'),\fp_\Om(\un\ze')]
\subset \un U_e
$
for all $\un\ze'\in\un U{}_{\un\ze}$,
where $\om'_{n',+}(\un\ze')$ is 
the intersection point of
$[\om'_{n'},\fp_\Om(\un\ze')]$
and
$
C_{\om'_{n'}}^{\varepsilon'}.
$
Define a deformation $Q^{\eps'}_{\un\ze,\un\ze'}$ 
$(\un\ze'\in \un U_e)$
of $P^{\eps'}_{\un\ze}$
by continuously varying
the arc of
$
C_{\om'_{n'}}^{\varepsilon'}
$
from $\om'_{n',-}$ to $\om'_{n',+}(\un\ze')$
and the line segment 
$[\om'_{n',+}(\un\ze'),\fp_\Om(\un\ze')]$
and fixing the other part of $P^{\eps'}_{\un\ze}$.
Then,
shrinking $\un U_{\un\ze}$ if necessary,
we find that $Q^{\eps'}_{\un\ze,\un\ze'}$ satisfies
$Q^{\eps'}_{\un\ze,\un\ze'}\in\Pi_\Om$
and
$L(Q^{\eps'}_{\un\ze,\un\ze'})<L(\un\ze')+2\eps$
for each $\un\ze'\in\un U{}_{\un\ze}$.
\end{enumerate}
Beware that,
when the edge $(v,v_\upa)$ is
the type i),
$Q^{\eps'}_{\un\ze,\un\ze'}$ is different from
$P^{\eps'}_{\un\ze'}$
for
$\un\ze\in\un \ell_{v\to v_\upa}$
and
$\un\ze'\in \un U_{\un\ze}\cap \un U_{v_\upa}$.
On the other hand,
$Q^{\eps'}_{\un\ze,\un\ze'}=P^{\eps'}_{\un\ze'}$
holds for
$\un\ze'\in \un U_{\un\ze}\cap \un U_{v}$.
When the edge $(v,v_\upa)$ is the type ii) or iii),
$Q^{\eps'}_{\un\ze,\un\ze'}=P^{\eps'}_{\un\ze'}$
holds for
$\un\ze\in\un \ell_{v\to v_\upa}$
and
$\un\ze'\in \un U_{\un\ze}$.

\medskip

Let $(X,\fp,\uO)$ be an $\Om$-endless Riemann surface.
For each $\un\ze\in X_\Om$,
take $\ga\in\Pi_\Om$ such that 
$\un\ga(1)=\un\ze$
and let $\uga_X$ be its lift on $X$.
Then, define a map $\fq:X_\Om\to X$
by $\fq(\un\ze)=\uga_X(1)$.
We now show the well-definedness of $\fq$.
%
%
For that purpose,
it suffices to prove the following
\begin{prp}\label{prp:3.10}
Let
$\ga_0,\ga_1\in\Pi_\Om$
such that
$\un\ga{}_0(1)=\un\ga{}_1(1)$.
Then,
there exists a continuous family $(H_s)_{s\in [0,1]}$
of $\Om$-allowed paths
satisfying the conditions
\begin{enumerate}
\item
$H_s(0)=0$
and
$H_s(1)=\ga_0(1)$
for all $s\in [0,1]$,
\item
$H_j=\ga_j$ for $j=0,1$.
\end{enumerate}

\end{prp}
The proof of Proposition \ref{prp:3.10}
is reduced to the following
\begin{lmm}\label{lmm:3.11}
For each $\ga\in\Pi_\Om$
and $\eps'>0$ sufficiently small,
there exists
a continuous family $(\ti H_s)_{s\in [0,1]}$
of $\Om$-allowed paths
satisfying the following conditions:
\begin{enumerate}
\item
$
L\big(\ti H_s\big)\leq
L(\ga_{|s})
$
and
$
\un{\ti H}{}_s(1)
=\un\ga(s)
$
for all
$s\in [0,1]$,

\item
$
\ti H_s=
P_{\un\ga(s)}^{\eps'}
$
for $s=0,1$.

\end{enumerate}

\end{lmm}

Notice that,
since $\un\ga(0)=\un0_\Om$,
$P_{\un\ga(0)}^{\eps'}$ is the constant map
$P_{\un\ga(0)}^{\eps'}=0$.

\begin{proof}
[Reduction of Proposition \ref{prp:3.10}
to Lemma \ref{lmm:3.11}]
For each $\ga\in\Pi_\Om$ and
$s\in(0,1]$,
define $H_s$ using $\ti H_s$ constructed
in Lemma \ref{lmm:3.11} as follows:
$$
H_s(t)\defeq 
\left\{
\begin{aligned}
&\ti H_s(t/s)
&\quad
\text{when}
\quad
&t\in[0,s],
\\
&\ga(t)
&\quad
\text{when}
\quad
&t\in[s,1].
\end{aligned}
\right.
$$
It extends continuously to $s=0$
and gives
a continuous family $( H_s)_{s\in [0,1]}$
of $\Om$-allowed paths
satisfying the assumption in Proposition \ref{prp:3.10}
with $\ga_0=\ga$ and $\ga_1=P_{\un\ga(1)}^{\eps'}$.

Now, 
let $\ga_0$ and $\ga_1$ be
the $\Om$-allowed paths 
satisfying the assumption in Proposition \ref{prp:3.10}.
Applying the above discussion to 
each of $\ga_0$ and $\ga_1$,
we obtain two families
of $\Om$-allowed paths
connecting them to $P_{\un\ga_0(1)}^{\eps'}$
and, concatenating the deformations at
$P_{\un\ga_0(1)}^{\eps'}$,
we obtain a deformation
$( H_s)_{s\in [0,1]}$
satisfying the conditions in Proposition \ref{prp:3.10}.
\end{proof}

\begin{proof}[Proof of Lemma \ref{lmm:3.11}]
Take $\de,L>0$ so that
$\ga\in\Pi_\Om^{\de,L}$.
We first show the following:
\begin{eq-text}\label{3.3}
When $\un\ga(t_0)\in \un U_{v\to (0,0)}$
for $t_0\in(0,1]$ and
$v=((0,0),(\om_2,\sigma_2))$,
the following estimate holds
for $t\in[t_0,1]$:
$$
L(\un\ga(t))+\sqrt{|\om_2|^2+\de^2}-|\om_2|
\leq L(\ga_{|t}).
$$
\end{eq-text}
Notice that,
since $\ga\in\Pi_\Om^{\de,L}$,
the length $L(\ga_{|t_0})$ of $\ga_{|t_0}$
must be longer than
that of the polygonal curve $C$
obtained by concatenating the line segments
$[0,\om_2+\de e^{i\theta}]$
and
$[\om_2+\de e^{i\theta},\ga(t_0)]$,
where
$
\theta=
\arg(\om_2)-\sigma_2\pi/2.
$
Then,
we find that,
for an arbitrary $\eps>0$,
taking $\eps'>0$ sufficiently small,
the path $\ti\ga^{\eps'}$
obtained by concatenating the paths
$P_{\un\ga(t_0)}^{\eps'}$ and $\ga|_{[t_0,1]}$
satisfies
$\ti\ga^{\eps'}\in\Pi_\Om$,
$\un{\ti\ga}^{\eps'}(t)=\un\ga(t)$
and
$
L(\ti\ga^{\eps'}_{|t})
\leq
L(\ti\ga^0_{|t})+\eps
$
for $t\in[t_0,1]$.
Therefore,
we have
$$
L(\un\ga(t))
\leq
L(\ti\ga^0_{|t})
\quad\text{for}\quad
t\in[t_0,1]
$$
Since $L(C)\geq \sqrt{|\om_2|^2+\de^2}+|\ga(t_0)-\om_2|$,
we find
$$
L(\ga_{|t})
=
L(\ti\ga^0_{|t})
+L(\ga_{|t_0})-L([0,\ga(t_0)])
\geq
L(\un\ga(t))
+\sqrt{|\om_2|^2+\de^2}-|\om_2|
$$
holds for $t\in[t_0,1]$,
and hence,
we obtain \eqref{3.3}.

Now,
we shall construct
$(H_s)_{s\in [0,1]}$.
Let $\eps>0$ be given.
We assign the path
$P_{\un\ga(t)}^{\eps'_t}$
$(\eps'_t>0)$
to each $t\in[0,1]$
and take a neighborhood
$\un U_{\un\ga(t)}$ of
$\un\ga(t)$ and
the deformation 
$Q_{\un\ga(t),\un\ze'}^{\eps'_t}$
$(\un\ze'\in \un U_{\un\ga(t)})$
of $P_{\un\ga(t)}^{\eps'_t}$
constructed in Lemma \ref{lmm:3.9}.
Then, we can cover $[0,1]$
by a finite number of intervals
$I_j=[a_j,b_j]$ $(j=1,2,\cdots,k)$
satisfying the following conditions:
\begin{enumerate}[--]
\item
The interior $I^\circ_j$ of $I_j$ satisfies
$I^\circ_{j_1}\cap I^\circ_{j_2}\neq \O$
when $|j_1-j_2|\leq 1$
and
$I_{j_1}\cap I_{j_2}= \O$
otherwise.

\item
There exists
$t_j\in I_j$
such that $t_j<t_{j+1}$  for $j=1,\cdots,k-1$
and
$\un\ga(I_j)\subset\un U_{\un\ga(t_j)}$.

\end{enumerate}
Notice that,
since $\un U_{\un\ga(t)}$ is taken
for each $t\in[0,1]$
so that it is contained in one of the charts
$\un U_v$ $(v\in V_\Om)$ or $\un U_e$
$(e\in E_\Om)$,
one of the followings holds:
\begin{enumerate}[--]

\item
$\un\ga(t_j)\in \un U_v$
and 
$
\un\ga(I_j)
\subset \un U_v
$
$(v\in V_\Om)$.

\item
$\un\ga(t_j)\in \un\ell_e$ 
and
$
\un\ga(I_j)
\subset \un U_e
$
$(e\in E_\Om)$.

\end{enumerate}
We set
$\dst
\eps'
=\min_j
\{\eps'_{t_j} \mid
\un\ga(t_j)\notin\un U_{(0,0)}
\}.
$
Then,
$P_{\un\ga(t_j)}^{\eps'}$
and its deformation
$Q_{\un\ga(t_j),\un\ze'}^{\eps'}$
$(\un\ze'\in \un U_{\un\ga(t_j)})$
also satisfy the conditions in
Lemma \ref{lmm:3.8} and Lemma \ref{lmm:3.9}.
Let $J_E\subset \{1,\cdots,k\}$
denote
the set of suffixes
satisfying the condition that
there exists $e\in E_\Om$ such that
$\un\ga(t_{j})\in \un\ell_e$
and let $j_0$ be the minimum of $J_E$.
Shrinking the neighborhood
$\un U_{\un\ga(t)}$ for each $t\in[0,1]$
at the first,
we may assume without loss of generality
that,
\begin{enumerate}[--]
\item
$|\ga(t)-\ga(t_j)|\leq \eps$
for $t\in I_j$
and $j=1,\cdots,k$,

\item
if $j$, $j+1\in J_E$,
there exists an edge $e\in E_\Om$ such that
$\un\ga(t_j)$, $\un\ga(t_{j+1})\in \un \ell_e$.

\end{enumerate}
Recall that,
from the construction of
$Q_{\un\ze,\un\ze'}^{\eps'}$,
\begin{equation*}
Q_{\un\ga(t_j),\un\ga(t)}^{\eps'}=
Q_{\un\ga(t_{j+1}),\un\ga(t)}^{\eps'}
\quad\text{for}\quad
t\in I_j\cap I_{j+1}
\end{equation*}
except for the cases where
there exists an edge $e=(v,v_\upa)\in E_\Om$
of the type i) such that
\begin{enumerate}[--]
\item
$
\un\ga(t_j)
\in\un U_{e}
$
and
$
\un\ga(t_{j+1})
\in \un U_{v_\upa},
$

\item
$
\un\ga(t_j)
\in\un U_{v_\upa}
$
and
$
\un\ga(t_{j+1})
\in \un U_{e}.
$

\end{enumerate}
In the first case,
the difference between
$
Q_{\un\ga(t_j),\un\ga(t)}^{\eps'}
$
and
$
Q_{\un\ga(t_{j+1}),\un\ga(t)}^{\eps'}
$
is the part from $\om^t(v_\upa)$ to $\ga(t)$,
where $\om^t(v_\upa)$ is
the intersection point of
$C_{\om(v_\upa)}^{\varepsilon'}$
and
$[\om(v_\upa),\ga(t)]$:
Let $\om_{e,i}$ $(i=0,\cdots,m+1)$
be the points  on the line segment
$[\om(v_\upa),\om(v)]$
satisfying the conditions
$(L_{e,i},\om_{e,i})\in\ov\cS_{\Om}$
and $L_{e,i}<L_{e,i+1}$,
where $L_{e,i}\defeq L(v_\upa)+|\om_{e,i}-\om(v_\upa)|$.
Then,
the part of
$
Q_{\un\ga(t_{j}),\un\ga(t)}^{\eps'}
$
from $\om^t(v_\upa)$ to $\ga(t)$
is given by concatenating
the arcs of 
$C_{\om_{e,i}}^{\varepsilon'}$
$(i=0,\cdots,m+1)$,
the intervals of
the line segment $[\om(v_\upa),\om(v)]$
and
$[\om^t(v),\ga(t)]$,
where $\om^t(v)$ is
the intersection point of
$C_{\om(v)}^{\varepsilon'}$
and
$[\om(v),\ga(t)]$.
(See Figure \ref{fig:4} (a).)
On the other hand,
$
Q_{\un\ga(t_{j+1}),\un\ga(t)}^{\eps'}
$
goes directly from $\om^t(v_\upa)$
to $\ga(t)$.
(See Figure \ref{fig:4} (d).)

Now,
let $\om^t_{i,+}$ (resp.\ $\om^t_{i,-}$)
be the intersection point of 
$C_{\om_{e,i}}^{\varepsilon'}$
and
$[\om^t(v_\upa),\om^t(v)]$
that is the closer to $\om^t(v)$
(resp.\ $\om^t(v_\upa)$).
While $t$ moves on $I_j\cap I_{j+1}$,
we first deform
the part of
$
Q_{\un\ga(t_{j}),\un\ga(t)}^{\eps'}
$
from $\om^t(v_\upa)$ to $\om^t(v)$
to the line segment $[\om^t(v_\upa),\om^t(v)]$
by shrinking the part of
$
Q_{\un\ga(t_{j}),\un\ga(t)}^{\eps'}
$
from $\om^t_{i,-}$ to $\om^t_{i,+}$
(resp.\ from $\om^t_{i,+}$ to $\om^t_{i+1,-}$)
to the line segment $[\om^t_{i,-},\om^t_{i,+}]$
(resp.\ $[\om^t_{i,+},\om^t_{i+1,-}]$) for each $i$.
(See Figure \ref{fig:4} (b) and (c).)
Then,
further shrinking the polygonal line
given by concatenating
$[\om^t(v_\upa),\om^t(v)]$
and
$[\om^t(v),\ga(t)]$
to the line segment
$[\om^t(v_\upa),\ga(t)]$,
we obtain a continuous family of $\Om$-allowed paths
$\big(\ti H_s\big)_{s\in [t_j,t_{j+1}]}$
satisfying the following conditions:
\begin{enumerate}[--]
\item
$
\ti H_s=
Q_{\un\ga(t_{j}),\un\ga(s)}^{\eps'}
$
when
$s\in [t_j,t_{j+1}]\setminus I_{j+1}$,

\item
$
\ti H_s=
Q_{\un\ga(t_{j+1}),\un\ga(s)}^{\eps'}
$
when
$s\in [t_j,t_{j+1}]\setminus I_j$,

\item
$
L\big(\ti H_s\big)\leq
L\big(Q_{\un\ga(t_{j}),\un\ga(s)}^{\eps'}\big)
$
and
$
\un{\ti H}{}_s(1)
=\un\ga(s)
$
when
$s\in I_j\cap I_{j+1}$.

\end{enumerate}
\begin{figure}[h]
\centering
\includegraphics[scale=0.4]{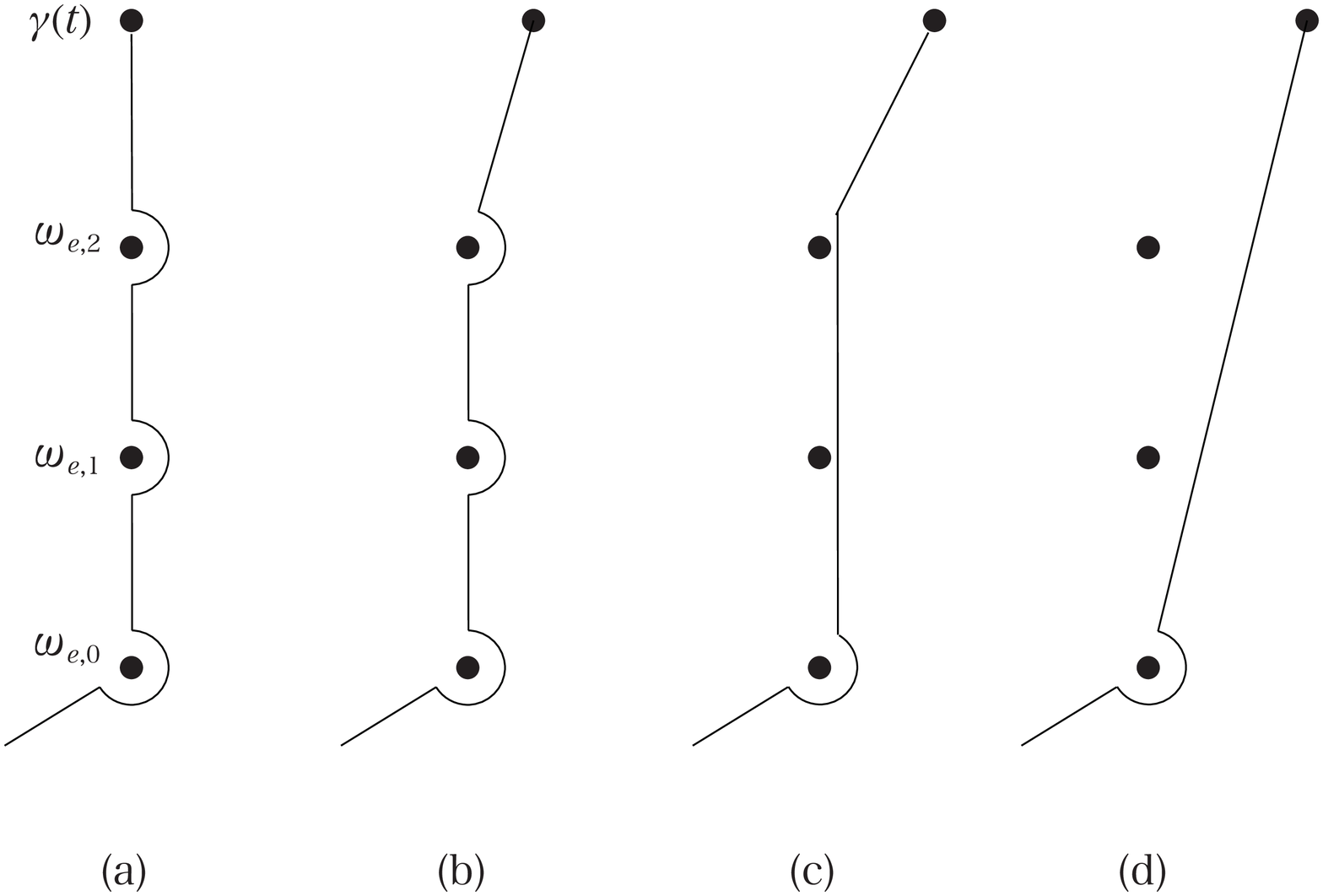}
\caption{}\label{fig:4}
\end{figure}

For the second case,
we can also construct
a continuous family of $\Om$-allowed paths
$\big(\ti H_s\big)_{s\in [t_j,t_{j+1}]}$
satisfying the first and the second conditions above
and
\begin{enumerate}[--]
\item
$
L\big(\ti H_s\big)\leq
L\big(Q_{\un\ga(t_{j+1}),\un\ga(s)}^{\eps'}\big)
$
and
$
\un{\ti H}{}_s(1)
=\un\ga(s)
$
when
$s\in I_j\cap I_{j+1}$.

\end{enumerate}
Then,
we can continuously extend $\ti H_s$
to $[0,1]$
by interpolating it by
$
Q_{\un\ga(t_{j}),\un\ga(s)}^{\eps'}
$
so that it satisfies
\begin{eq-text}\label{3.4}
$\dst
L\big(\ti H_s\big)\leq
\max_j\big\{
L\big(Q_{\un\ga(t_{j}),\un\ga(s)}^{\eps'}\big)
\mid
s \in I_j
\big\}
$
and
$
\un{\ti H}{}_s(1)
=\un\ga(s)
$
for all
$s\in [0,1]$.
\end{eq-text}
Since $I_{j_0}$ is taken so that
$|\ga(t)-\ga(t_{j_0})|\leq \eps$
holds on $I_{j_0}$,
applying \eqref{3.3} with $t_0=t_{j_0}$,
we have the following estimates:
$$
L(\un\ga(t))+\sqrt{|\om_2|^2+\de^2}-|\om_2|-\eps
\leq L(\ga_{|t})
\quad
\text{for}
\quad
t\in[a_{j_0},1].
$$
On the other hand,
since $\un\ga(t)\in \un U_{(0,0)}$ for $t\in [0,a_{j_0}]$,
we find
$
L\big(Q_{\un\ga(t_{j}),\un\ga(t)}^{\eps'}\big)=
L(\un\ga(t))
$
holds 
for $t\in I_j$ and $j<j_0$
from the construction of $Q_{\un\ze,\un\ze'}^{\eps'}$.
Therefore,
taking $\eps>0$ sufficiently small so that 
$$
3\eps\leq
\sqrt{|\om_2|^2+\de^2}-|\om_2|,
$$
we obtain the following estimates
from Lemma \ref{lmm:3.8}
and \eqref{3.4}:
$$
L\big(\ti H_s\big)\leq
L(\ga_{|s})
\quad
\text{for}
\quad
s\in[0,1].
$$
Finally,
from the construction of
$\ti H_s$,
we find that $\ti H_s$ satisfies
$
\ti H_s=
P_{\un\ga(s)}^{\eps'}
$
for $s=0,1$.
\end{proof}

Since $\fp_\Om=\fp\circ\fq$ and
$\fp$ is isomorphic near $\uO$,
all the maps $\fq:X_\Om\to X$
must coincide near $\uO_\Om$,
and hence,
uniqueness of $\fq$ follows from
the uniqueness of the analytical continuation
of $\fq$.
Finally,
$X_\Om$ is unique up to isomorphism
because $X_\Om$ is an initial object in
the category of $\Om$-endless Riemann surfaces.

\subsection{Supplement to the properties of $X_\Om$}

Let $\gO_X$ denote the sheaf of holomorphic functions on
a Riemann surface $X$
and consider the natural morphism
$\fp_\Om^* \col \fp\ii_\Om\gO_{\C} \to \gO_{X_\Om}$
induced by 
$
\fp_\Om:X_\Om\to\C.
$
Since $X_\Om$ is simply connected,
we obtain the following:
\begin{prp}
Let $\hat\ph \in  \gO_{\C,0}$. 
Then the followings are equivalent:
\begin{enumerate}[{\rm i)}]
\item
$\hat\ph\in\gO_{\C,0}$ is $\Om$-continuable,
\item
$\fp_\Om^*\hat\ph \in \gO_{X_{\Om},\uO_\Om}$ can be analytically continued
along any path on $X_{\Om}$,
\item
$\fp_\Om^*\hat\ph \in \gO_{X_{\Om},\uO_\Om}$ can be extended to 
$\Ga(X_{\Om},\gO_{X_{\Om}})$.
\end{enumerate}
\end{prp}

Therefore, we find
\[
\fp_\Om^* \col
\hat\gR_\Om \isom \Ga(X_{\Om},\gO_{X_{\Om}}).
\]

\begin{nota}
For $L,\de>0$, 
using $\Pi_\Om^{\de,L}$ of~\eqref{eq:defPideL},
we define a compact subset
$K_{\Om}^{\de,L}$ of $X_\Om$ by
\beglab{eq:defKOmdeL}
K_{\Om}^{\de,L} \defeq \big\{\, \un\ze\in X_\Om \mid 
\;\text{$\exists \ga\in\Pi_\Om^{\de,L}$ such that 
$\un\ze = \uga(1)$}
\,\big\}.
\edla
%
%
\end{nota}
Notice that
$X_\Om$ is exhausted by $(K_\Om^{\de,L})_{\de,L>0}$.
Therefore, the family of
seminorms $\|\cdot\|_{\Om}^{\de,L}$ $(\de,L>0)$
defined by
$$
\|\, \hat f\,
\|_{\Om}^{\de,L}
\defeq \sup_{\un\ze\in K_\Om^{\de,L}}
| \hat f(\un\ze)|
\quad
\text{for}
\ens
\hat f\in\Ga(X_{\Om},\gO_{X_{\Om}})
$$
induces
a structure of
Fr\'echet space on 
$\Ga(X_{\Om},\gO_{X_{\Om}}).$

\begin{dfn}\label{dfn:3.8}
We introduce a structure of Fr\'echet space on 
$\ti\gR_\Om$ by a family of seminorms
$\|\cdot\|_{\Om}^{\de,L}$ $(\de,L>0)$
defined by
$$
\|\, \ti\ph\,
\|_{\Om}^{\de,L}
\defeq 
|\ph_0|+
\|\,\fp_\Om^*\hat\ph\,
\|_{\Om}^{\de,L}
\quad
\text{for}
\quad
\ti\ph\in\ti\gR_\Om,
$$
where 
$
\cB(\ti\ph)
=\ph_0\de+\hat\ph
\in\C\de \oplus \hat\gR_\Om.
$
\end{dfn}

Let $\Om'$ be a \dfs\ such that
$\Om\subset\Om'$.
Since $\Pi_{\Om'}\subset\Pi_{\Om}$,
$X_{\Om}$ is $\Om'$-endless.
Therefore,
Theorem~\ref{lemuniversalXOm} yields a morphism
$$
\fq:(X_{\Om'},\fp_{\Om'},\uO_{\Om'})
\to (X_{\Om},\fp_\Om,\uO_\Om),
$$
which induces a morphism 
$\fq^*:\fq\ii\gO_{X_{\Om}}\to\gO_{X_{\Om'}}$.
Since 
$
\fq(K_{\Om'}^{\de,L})\subset K_\Om^{\de,L},
$
we have
$$
\|\,\fq^* \hat f\,
\|_{\Om'}^{\de,L}
\leq
\|\, \hat f\,
\|_{\Om}^{\de,L}
\quad
\text{for}
\ens
\hat f\in\Ga(X_{\Om},\gO_{X_{\Om}}),
$$
and hence,
$$
\|\, \ti\ph\,
\|_{\Om'}^{\de,L}
\leq
\|\, \ti\ph\,
\|_{\Om}^{\de,L}
\quad
\text{for}
\ens
\ti\ph\in\ti\gR_{\Om}.
$$

In view of Theorem~\ref{thm:4.9} below,
the product map
$ 
\ti\gR_{\Om}\times\ti\gR_{\Om'}
\to\ti\gR_{\Om*\Om'}
$ 
is continuous
and hence,
when $\Om*\Om=\Om$,
$\ti\gR_{\Om}$ is a Fr\'echet algebra.


\subsection{The endless Riemann surface associated with a \ddfs}


In this section, we discuss the construction of
the endless Riemann surfaces associated with an arbitrary \ddfs~$\Om$.
Let us first define the skeleton of $\Om$:
\begin{dfn}
Let
$
V_\Om\subset
\bigcup_{n=1}^{\infty}
(\C\times\Z)^n
$
be the set of vertices
$$
v\defeq 
((\om_1,\sigma_1),\cdots,(\om_n,\sigma_n))
\in
(\C\times\Z)^n
$$
that satisfy the conditions
1) and 2) in Definition \ref{dfn:3.3} and
\begin{enumerate}
\item[3')]
$\big(M_j(v),L_j(v),\om_j\big)\in \ov\cS_\Om$
for $j=2,\cdots,n$,
\end{enumerate}
with
$
L_{j}(v)\defeq \sum_{i=1}^{j-1}
|\om_{i+1}-\om_i|
$
$(j=2,\cdots,n)$,
$$
M_j(v)\defeq
\left\{ \begin{aligned}
&0
&(j=2),
\\
&\sum_{i=2}^{j-1}\big(A_i(v)+2\pi(|\sigma_i|-1)\big)
&(j=3,\cdots,n),
\end{aligned} \right.
$$
and
$$
A_i(v)\defeq
\left\{ \begin{aligned}
&|\theta_i|
&\text{ if }\theta_i\sigma_i\geq0,
\\
&2\pi-|\theta_i|
&\text{ if }\theta_i\sigma_i<0,
\end{aligned} \right.
$$
where
$$
\theta_i:=\arg\frac{\om_{i+1}-\om_i}{\om_i-\om_{i-1}}
$$
is taken so that $\theta_i\in(-\pi,\pi]$.
Let $E_\Om\subset V_\Om\times V_\Om$ be the set of edges
$e=(v',v)$ that satisfy one of the conditions
i) $\sim$ iii) in Definition \ref{dfn:3.3}.
We denote
the directed tree diagram $(V_\Om,E_\Om)$
by $Sk_\Om$
and call it
\emph{skeleton} of $\Om$.
\end{dfn}

Now, assigning a cut plane $U_v$ (resp.\ an open set $U_e$)
to each $v\in V_\Om$ (resp.\ each $e\in E_\Om$ of type i))
defined by totally the same way with Section \ref{sec:3.1}
and patching them as in Section \ref{sec:3.1},
we obtain an initial object
$(X_{\Om},\fp_\Om,\uO_\Om)$
in the category of $\Om$-endless Riemann surfaces
associated with a \ddfs\ $\Om$.
We denote the lift of $\ga\in\Pi_\Om\dv$ on $X_\Om$ by $\un\ga$.



\section{
Estimates for the analytic continuation of iterated convolutions
}
\label{sec:4}

In this section, our aim is to prove the following theorem,
which is the analytical core of our study of the convolution product
of endlessly continuable functions.

\begin{thm}\label{thm:4.7}
Let $\de,L>0$ be reals.
Then there exist $c,\de'>0$ such that,
for every \dfs~$\Om$ such that $\Om_{4\de}=\O$, for every integer $n\ge1$ and for
every
$\hat f_1, \ldots, \hat f_n \in \hat\gR_\Om$,
the function $1*\hat f_{1}*\cdots *\hat f_{n}$ (which is known to belong to $\hat\gR_{\Om^{*n}}$)
satisfies
\beglab{4.6}
\big|\,
\fp_{\Om^{*n}}^* \big( 1*\hat f_{1}*\cdots *\hat f_{n}\big )
(\un\ze)
\big|
\leq
\frac{c ^n}{n!}
%
\ \sup_{L_1+\cdots+L_n=L}
 \big\|\, \fp_\Om^{*} \hat f_1 \big\|_{\Om}^{\de',L_1}
\, \cdots
\big\|\, \fp_\Om^{*} \hat f_n \big\|_{\Om}^{\de',L_n}
\quad\text{for} \ens \un\ze \in K_{\Om^{*n}}^{\de,L}
\edla
(with notation~\eqref{eq:defKOmdeL}).
\end{thm}

Using the Cauchy inequality, the identity
$ 
\frac{\dd\,}{\dd\ze}(1*\hat f_{1}*\cdots *\hat f_{n})
=\hat f_{1}*\cdots *\hat f_{n}
$ 
and the inverse Borel transform,
one easily deduces the following 

\begin{crl}\label{crl:4.7'}
Let $\de,L>0$ be reals.
Then there exist $c,\de',L'>0$ such that,
for every \dfs~$\Om$ such that $\Om_{4\de}=\O$, for every integer $n\ge1$ and for
every
$ \ti f_1, \ldots,  \ti f_n \in \ti\gR_\Om$ without constant term,
the formal series $\ti f_{1}\cdots  \ti f_{n}$
(which is known to belong to $\ti\gR_{\Om^{*n}}$)
satisfies
$$
\big\|\,
\ti f_{1}\cdots  \ti f_{n}
\big\|_{\Om^{*n}}^{\de,L}
\leq
\frac{c ^{n+1}}{n!}\,
 \big\|\,  \ti f_1 \big\|_{\Om}^{\de',L'}
\, \cdots
\big\|\,  \ti f_n \big\|_{\Om}^{\de',L'}.
$$
\end{crl}

In fact, one can cover the case $\hat f_1\in\hat\gR_{\Om_1}$, \ldots,
$\hat f_n\in\hat\gR_{\Om_n}$ with different \dfs's
$\Om_1,\ldots,\Om_n$ as well---see Theorem~\ref{thm:4.9}---, but we
only give details for the case of one \dfs\ so as to lighten the presentation.


%
\subsection{Notations and preliminaries}
%

We fix an integer $n\ge1$ and a \dfs~$\Om$.
In view of Remark~\ref{rem:wlogupclos}, without loss of generality, we can suppose that~$\Om$ coincides with
its upper closure:
\beglab{eq:wlogOm}
\Om=\ti\Om.
\edla
Let $\rho>0$ be such that $\Om_{3\rho}=\O$. We set
\[ U \defeq \{\, \ze\in\C \mid \abs{\ze} < 3\rho \,\}. \]
For each $\ze\in U$, the path $\ga_\ze \col t\in[0,1] \mapsto t\ze$ is
$\Om$-allowed and hence has a lift $\uga_\ze$ on~$X_\Om$
starting at~$\uO_\Om$.  
Then $\gL(\ze) \defeq \uga_\ze(1) $ defines a holomorphic function
on~$U$ and induces an isomorphism
\beglab{eq:defcLuU}
\gL \col U \isom \uU, 
\qquad\text{where}\ens
\uU \defeq \gL(U) \subset X_\Om,
\edla
such that $\fp_\Om \circ \gL = \ID$.


%
Let us denote by~$\De_n$ the $n$-dimensional
simplex
\[
\De_n \defeq \{\, (s_1,\ldots,s_n)\in\R_{\geq0}^n \mid 
s_1 +\cdots + s_n \le 1 \,\}
\]
with the standard orientation, and by $[\De_n]\in\gE_n(\R^n)$ the corresponding integration
current.
For $\ze \in U$, we define a map $\gD(\ze)$ on a neighbourhood
of~$\De_n$ in~$\R^n$ by
\[
\gD(\ze) \col
\vec s = (s_1,\ldots,s_n) \mapsto 
\gD(\ze,\vec s\,) \defeq \big( \gL(s_1\ze),\ldots,\gL(s_n\ze) \big) \in
\uU^n \subset X_\Om^n
\]
and denote by $\gD(\ze)_\# [\De_n] \in \gE_n(X_\Om^n)$ the push-forward
of~$[\De_n]$ by~$\gD(\ze)$.
(See \cite{NLresur} for the notations and notions
related to integration currents.)


As in \cite{NLresur}, our starting point will be
\begin{lmm}   \label{lmn:formbeta}
Let $\hat f_1, \ldots, \hat f_n \in \hat\gR_\Om$ and
$ 
\be \defeq (\fp_\Om^*\hat f_1)\big(\un\ze_1\big) \cdots (\fp_\Om^*\hat f_n)\big(\un\ze_n\big) \,
\dd\un\ze_1 \wedge \cdots \wedge \dd\un\ze_n,
$, 
where we denote by $\dd\un\ze_1 \wedge \cdots \wedge \dd\un\ze_n$ the pullback
by $\fp_\Om^{\otimes n} \col X_\Om^n \to \C^n$
of the $n$-form $\dd\ze_1 \wedge \cdots \wedge \dd\ze_n$.
Then
\[
1*\hat f_1*\cdots *\hat f_n(\ze) =
\gD(\ze)_\# [\De_n](\be)
\quad \text{for $\ze \in U$}.
\]
\end{lmm}

\begin{proof}
This is just another way of writing the formula
\beglab{eq:initformul}
1*\hat f_1*\cdots *\hat f_n(\ze) =
\ze^n\int_{\De_n} 
\hat f_1(\ze s_1) \cdots \hat f_n(\ze s_n)
\ \dd s_1 \cdots \dd s_n.
\edla
See \cite{NLresur} for the details.
\end{proof}

\begin{nota}
We set
\begin{align}
\label{eq:defcNze}
\cN(\ze) &\defeq \big\{ \big(\un\ze_1,\ldots,\un\ze_n\big) \in X_\Om^n \mid 
\fp_\Om\big(\un\ze_1\big) + \cdots + \fp_\Om\big(\un\ze_n\big) = \ze \big\}
\ens\text{for $\ze\in\C$,}
\\[1ex]
\label{eq:defcNi}
\cN_j &\defeq \big\{ \big(\un\ze_1,\ldots,\un\ze_n\big) \in X_\Om^n \mid \un\ze_j = \uO_\Om \big\}
\ens\text{for $1\le j\le n$.}
\end{align}
\end{nota}

%
\subsection{$\ga$-adapted deformations of the identity}
%

%
Let us consider a path $\ga \col [0,1] \to \C$ in $\Pi_{\Om^{*n}}$
for which there exists $a\in(0,1)$ such that 
\begin{eq-text}\label{4.1}
$\ga(t) = \frac{t}{a} \ga(a)$ for $t\in[0,a]$, \quad
$\abs{\ga(a)} = \rho$, \quad
$\ga|_{[a,1]}$ is $C^1$.
\end{eq-text}

We now introduce the notion of $\ga$-adapted deformation of the
identity, which is a slight generalization of the $\ga$-adapted
origin-fixing isotopies which appear in \cite[Def.~5.1]{NLresur}.
%
%
\begin{dfn}
A \emph{$\ga$-adapted deformation of the identity}
is a family $(\Psi_t)_{t\in[a,1]}$ of maps 
\[\Psi_t \col \uV \to X_\Om^n, \quad\text{for $t\in[a,1]$,}\]
where $\uV \defeq \gD\big( \ga(a) \big)(\De_n) \subset X_\Om^n$,
such that $\Psi_a = \ID$, 
the map
$\big(t,\vuze\,\big) \in [a,1] \times \uV \mapsto
\Psi_t\big(\vuze\,\big) \in X_\Om^n$
is locally Lipschitz,
and for any $t\in [a,1]$ and $j=1,\ldots,n$,
\beglab{eq:Psitinclus}
\Psi_t\big( \uV \cap
\cN\big(\ga(a)\big)\big) 
\subset \cN\big(\ga(t)\big), \qquad 
\Psi_t\big( \uV \cap \cN_j\big) 
\subset \cN_j
\edla
(with the notations \eqref{eq:defcNze}--\eqref{eq:defcNi}).
\end{dfn}


Let~$\uga$ denote the lift of~$\ga$ in~$X_\Om$ starting at~$\uO_\Om$.
The analytical continuation along~$\uga$ of a convolution product can
be obtained as follows:

\begin{prp}[\cite{NLresur}]   \label{prp:DeformInteg}
If $(\Psi_t)_{t\in[a,1]}$ is a $\ga$-adapted deformation of the
identity, then
\beglab{eqconthatgugat}
\fp_{\Om^{*n}}^* \big( 1*\hat f_{1}*\cdots *\hat f_{n}\big )\big( \uga(t) \big) = 
\big( \Psi_t \circ \gD\big( \ga(a) \big) \big)_\# [\De_n](\be)
\qquad\text{for $t\in[a,1]$}
\edla
for any $\hat f_1, \ldots, \hat f_n \in \hat\gR_\Om$,
with~$\be$ as in Lemma~\ref{lmn:formbeta}.
\end{prp}


\begin{proof}
See the proof of \cite[Prop.~5.2]{NLresur}.
\end{proof}


Note that the \rhs\ of~\eqref{eqconthatgugat} must be interpreted as
\beglab{eqconthatgugatbis}
\int_{\De_n} 
(\fp_\Om^*\hat f_1)\big(\un\ze_1^t\big) \cdots (\fp_\Om^*\hat f_n)\big(\un\ze_n^t\big)
\Det\bigg[ \frac{\pa\ze_i^t}{\pa s_j}\bigg]_{1\le i,j\le n}
\, \dd s_1 \cdots \dd s_n
\edla
with the notation
\beglab{eq:notunzeitzeit}
\big( \un\ze_1^t, \ldots, \un\ze_n^t \big) \defeq
\Psi_t\circ \gD\big( \ga(a) \big), \qquad
\ze_i^t \defeq \fp_\Om \circ \un\ze_i^t
\ens\text{for $1\le i\le n$}
\edla
(each function $\ze_i^t$ is Lipschitz on~$\De_n$ and
Rademacher's theorem ensures that it is differentiable almost
everywhere on~$\De_n$, with bounded partial derivatives).

The following is the key estimate:

\begin{thm}\label{thm:4.6}
Let $\de\in(0,\rho)$ and $L>0$.
Let $\ga \in \Pi_{\Om^{*n}}^{\de,L}$ satisfy~\eqref{4.1} and let
\begin{equation}\label{4.5}
\de'(t) \defeq  \rho\, \ee^{-2\sqrt{2}\delta^{-1}L( \restr{\ga}{[a,t]} )},
\qquad
c(t) \defeq 
\rho \,\ee^{3\de\ii L( \restr{\ga}{[a,t]} )}
\qquad\text{for $t\in[a,1]$.}
\end{equation}
  Then there exists a $\ga$-adapted deformation of the identity
  $(\Psi_t)_{t\in[a,1]}$ such that
\begin{equation}\label{4.3}
\Psi_{t} \circ \gD\big( \ga(a) \big)(\De_n)
\subset
\bigcup_{L_1+\cdots+L_n=L(\ga_{|t})}
K_{\Omega}^{\de'(t),L_{1}}
\times\cdots\times
K_{\Omega}^{\de'(t),L_{n}}
\quad\text{for $t\in[a,1]$.}
\end{equation}
Further, with the notation~\eqref{eq:notunzeitzeit},
the partial derivatives $\pa\ze_i^t/\pa s_j$ satisfy
\beglab{4.4}
\Big|
\Det\Big[ \frac{\pa\ze_i^t}{\pa s_j}\Big]_{1\le i,j\le n}
\Big|
\leq
\big(c(t)\big)^n
\quad\text{a.e.\ on~$\De_n$}
\edla
for each $t\in[a,1]$.
\end{thm}


\begin{proof}[Proof that Theorem~\ref{thm:4.6} implies Theorem~\ref{thm:4.7}]
Let $\de,L>0$. We will show that~\eqref{4.6} holds with
\[
\de' \defeq \min\big\{ \de, \rho\, \ee^{-4\sqrt{2}(1+\de\ii L)} \big\},
\quad
c \defeq \max\big\{ 2\rho, \rho \,\ee^{6(1+\de\ii L)} \big\},
\quad
\text{where}\ens \rho \defeq \tfrac{4}{3}\de.
\]
Let~$\Om$ be a \dfs\ such that $\Om_{4\de}=\O$. 
Without loss of generality we may suppose that $\Om=\ti\Om$.

In view of formula~\eqref{eq:initformul}, the inequality~\eqref{4.6}
holds for $\uze \in K_{\Om^{*n}}^{\de,L} \cap \uU$, where~$\uU$ is defined
by~\eqref{eq:defcLuU},
because the Lebesgue measure of~$\De_n$ is $1/n!$.

Let $\un\ze\in K_{\Om^{*n}}^{\de,L}\setminus\uU$.
We can write $\un\ze=\uga(1)$ with $\ga\in\Pi_{\Om^{*n}}^{\de,L}$, 
assuming without loss of generality that the first two conditions
in~\eqref{4.1} hold. 
If the third condition in~\eqref{4.1} does not hold, \ie
if $\ga|_{[a,1]}$ is not $C^1$, then we use a sequence of paths
$\ga_k \in \Pi_{\Om^{*n}}^{\de/2,L+\de}$ such that
$\ga_k|_{[0,a]} =\ga|_{[0,a]}$, $\ga_k(1)=\ga(1)$,
$\ga_k|_{[a,1]}$ is $C^1$
and $\sup_{t\in[a,1]} |\ga(t)-\ga_k(t)| \to 0$ as $k\to\infty$;
for~$k$ large enough one has $\un{\ga_k}(1)=\un\ze$, thus one then can
replace~$\ga$ by~$\ga_k$.
Hence we can assume that~\eqref{4.1} holds. 
%
%
%
Let~$(\Psi_t)_{[t\in[a,1]]}$ denote the $\ga$-adapted deformation of
the identity provided by Theorem~\ref{thm:4.6}, possibly with
$(\de,L)$ replaced by $(\de/2,L+\de)$.
Proposition~\ref{prp:DeformInteg} shows that, 
for $\hat f_1, \ldots, \hat f_n \in \hat\gR_\Om$,
$\fp_{\Om^{*n}}^* \big( 1*\hat f_{1}*\cdots *\hat f_{n}\big )(\uze)$
can be written as~\eqref{eqconthatgugatbis} with $t=1$,
and \eqref{4.3}--\eqref{4.4} then show that~\eqref{4.6} holds because
$\de'(t)\ge\de'$ and $c(1)\le c$.
Therefore, \eqref{4.6} holds on $K_{\Om^{*n}}^{\de,L}\setminus\uU$
too.
\end{proof}

In fact, in view of the proof of Theorem~\ref{thm:4.6} given below,
one can give the following generalization of Theorem \ref{thm:4.7}:

\begin{thm}\label{thm:4.9}
Let $\de,L$ be positive real numbers.
Then there exist positive constants~$c$ and~$\de'$ such that, 
for every integer $n\ge1$ and for all \dfs\ 
$\Om_1,\ldots,\Om_n$ with $\Om_{j,4\de}=\O$
$(j=1,\cdots,n)$ 
and $\hat f_1\in \hat\gR_{\Om_1}, \ldots, \hat f_n \in \hat\gR_{\Om_n}$,
the function $1*\hat f_{1}*\cdots *\hat f_{n}$ belongs to $\hat\gR_{\Om}$,
where $\Om \defeq \Om_1*\cdots*\Om_n$,
and 
\beglab{4.8}
\big|
\fp_{\Om}^* \big( 1*\hat f_{1}*\cdots *\hat f_{n}\big )
(\un\ze)
\big|
\leq
\frac{c ^n}{n!}
%
\ \sup_{L_1+\cdots+L_n=L}
\big\|\, \fp_{\Om_1}^{*} \hat f_1 \big\|_{\Om_1}^{\de',L_1}
\, \cdots
\big\|\, \fp_{\Om_n}^{*} \hat f_n \big\|_{\Om_n}^{\de',L_n}
\quad\text{for} \ens \un\ze \in K_{\Om}^{\de,L}.
\edla
\end{thm}

%
\subsection{Proof of Theorem~\ref{thm:4.6}}
%

We suppose that we are given $n\ge1$, $\rho>0$, a \dfs~$\Om$
such that $\Om=\ti\Om$ and $\Om_{3\rho}=\O$,
and $\ga \in \Pi_{\Om^{*n}}^{\de,L}$ satisfying~\eqref{4.1} with $\de\in(0,\rho)$ and $L>0$.

We set $\tilde{\ga}(t)\defeq \big( L(\ga_{|t}),\ga(t) \big)$ and
define functions
\[
\eta:\R\times\C \to\Rp, \qquad \qquad
%
%
D:[a,1]\times (\R\times\C)^n\to\Rp
\]
by the formulas
\beglab{eq:defetaD}
\eta(v) \defeq \dist\big( v, \{ (0,0) \} \cup \cS_\Om \big),
\quad
D\big(t,\vec{v}\,\big) \defeq
\eta(v_1) + \cdots + \eta(v_n) + 
\big| \tilde{\ga}(t) -(v_1+\cdots+v_n )\big|,
\edla
where $|\,\cdot\,|$ is the Euclidean norm in 
$\R\times\C\simeq \R^3$.
The assumptions $\Om=\ti\Om$ and $\ga\in\Pi_{\Om^{*n}}^{\de,L}$ yield


\begin{lmm}\label{lmm:4.10}
The function $D$ satisfies
\begin{equation}\label{4.9}
D\geq\de
\quad\text{on}
\ens
[a,1]\times(\R\times\C)^n.
\end{equation}
\end{lmm}


\begin{proof}
Let $(t,\vec{v}\,) \in [a,1]\times(\R\times\C)^n$.
For each $j\in\{1,\ldots,n\}$, pick $u_j\in\{ (0,0) \} \cup \ov\cS_\Om$
so that $\eta(v_j)=|v_j-u_j|$, and let $u\defeq u_1+\cdots+u_n$.
Either all of the $u_j$'s equal~$(0,0)$, in which case $u=(0,0)$ too,
or $u=(\la,\om)$ is a non-trivial sum of at most~$n$ points of the form
$u_j=(\la_j,\om_j)\in\ov\cS_\Om$, in which case we have in fact
$\om_j\in\Om_{\la_j}$ because of Lemma~\ref{lemclosOm} and the
assuption $\ti\Om=\Om$, hence~\eqref{eq:defsumdfs} then yields
$\om\in\Om^{*n}_\la$.
We thus find
\beglab{ineq:Dtvecv} 
D\big(t,\vec{v}\,\big) 
%
= |v_1-u_1| + \cdots + |v_n-u_n| + 
\big| \tilde{\ga}(t) -(v_1+\cdots+v_n )\big|
%
\geq | \tilde{\ga}(t) - u |
\edla 
with $u\in\{ (0,0) \} \cup\cS_{\Om^{*n}}$.

If $u=(0,0)$, then $D\big(t,\vec{v}\,\big) \ge \big| \ti\ga(t)\big|
\ge L(\ga_{|t}) \ge \rho \ge \de$ because $t\ge a$.

Otherwise, $u\in\cS_{\Om^{*n}}$ and~\eqref{ineq:Dtvecv} shows that
$D\big(t,\vec{v}\,\big) \ge\de$
because $\ga\in\Pi_{\Om^{*n}}^{\de,L}$.
\end{proof}


Since~$D$ never vanishes, we can define a non-autonomous vector field 
\[
(t,\vec v\,) \in [a,1]\times(\R\times\C)^n \mapsto
\vec X(t,\vec v\,) \in T_{\vec v}\big( (\R\times\C)^n \big) 
\simeq (\R\times\C)^n
\]
by the formulas
\beglab{eqdefvecX}
\vec X(t,\vec v\,) =  \left| \begin{aligned}
X_1(t,\vec v\,) &\defeq \frac{\eta(v_1)}{D(t,\vec v\,)} \ti\ga'(t) 
\\
& \qquad \vdots \\[1ex]
X_n(t,\vec v\,) &\defeq \frac{\eta(v_n)}{D(t,\vec v\,)}  \ti\ga'(t)
\end{aligned} \right.
\edla
Note that $\ti\ga'(t) = \big( \abs{\ga'(t)}, \ga'(t) \big)$.

The functions $X_j \col [a,1]\times (\R\times\C)^n \to \R\times\C$ are
locally Lipschitz, thus we can apply the Cauchy-Lipschitz theorem on
the existence and uniqueness of solutions to differential equations
and get a locally Lipschitz flow map
\beglab{eq:defflowmap}
(t^*,t,\vec v\,) \in [a,1]\times[a,1]\times (\R\times\C)^n 
\mapsto \Phi^{t^*,t}(\vec v\,) \in (\R\times\C)^n
\edla
(value at time~$t$ of the unique maximal solution to $\dd\vec v/dt =
\vec X(t,\vec v\,)$ whose value at time~$t^*$ is~$\vec v$\,).
We construct a $\ga$-adapted deformation of the identity out of the
flow map as follows:


\begin{prp}   \label{prp:constructPsit}
Let $\vuze = \big( \gL(\ze_1),\ldots, \gL(\ze_n) \big) \in \uV$,
\ie $\ze_j = s_j \ga(a)$ with $(s_1,\ldots,s_n) \in \De_n$.
We define $\vec v \defeq 
\big( (\abs{\ze_1},\ze_1), \ldots, (\abs{\ze_n},\ze_n) \big) \in (\R\times\C)^n$
and $\Ga =
(\ti\ga_{1}, \ldots,\ti\ga_{n}) 
\col [0,1] \to (\R\times\C)^n$ by
\[
t \in [0,a] \ens\Rightarrow\ens 
\Ga(t) \defeq \big(\tfrac{t}{a} (\abs{\ze_1},\ze_1), \ldots, \tfrac{t}{a} (\abs{\ze_n},\ze_n) \big),
\qquad
t \in [a,1] \ens\Rightarrow\ens 
\Ga(t) \defeq \Phi^{a,t}(\vec v\,).
\]
Then, for each $j\in\{1,\ldots,n\}$, $\ti\ga_{j}$ is a path $[0,1]\to\R\times\C$ whose
$\C$-projection $\ga_{j}$ belongs to~$\Pi_\Om$, and the formula
\beglab{eq:defPsit}
\Psi_t\big(\vuze\,\big) \defeq \big(
\uga_{1}(t), \ldots, \uga_{n}(t)
\big) \in X_\Om^n
\quad \text{for $t\in[a,1]$.}
\edla
defines a $\ga$-adapted deformation of the identity.
\end{prp}


\begin{proof}
We first prove that 
$\ga_{1}, \ldots, \ga_{n} \in \Pi_\Om$.
In view of~\eqref{eqtextcharacOmalltiga}, we just need to check that, for each
$j\in\{1,\ldots,n\}$,
the path $\ti\ga_{j} = (\la_{j},\ga_{j})$ satisfies
\beglab{eq:justneed}
t\in[0,1] \quad\Rightarrow\quad
\ti\ga_{j}(t) \in \cM_\Om
\ens\text{and}\ens
{\dd \la_{j}}/{\dd t} = 
\abs{{\dd \ga_{j}}/{\dd t}}.
\edla

Since $\ze_j \in U$ and $\ga_{j}(t) = \tfrac{t}{a}\ze_j$ for
$t\in[0,a]$, the property~\eqref{eq:justneed} holds for $t\in[0,a]$.

For $t\in[a,1]$, the second property in~\eqref{eq:justneed} follows
from the fact that the $\R$-projection of
$X_j(t,\vec v\,) \in \R\times \C$ coincides with the modulus of its
$\C$-projection.

Since $\big(\ti\ga_{1}(t), \ldots,\ti\ga_{n}(t)\big) 
= \Phi^{a,t}\big(\ti\ga_{1}(a), \ldots,\ti\ga_{n}(a)\big)$
and the first property in~\eqref{eq:justneed} holds at $t=a$,
the first property in~\eqref{eq:justneed} for $t\in[a,1]$ is a consequence of the inclusion
\beglab{inclusPhiatcMn}
\Phi^{a,t}\big(\cM_\Om^n\big) \subset \cM_\Om^n,
\edla
which can itself be checked as follows:
suppose $\vec v^* \in (\R\times\C)^n\setminus\cM_\Om^n$, then it has at
least one component~$v^*_j$ in~$\ov\cS_\Om$ and, in view of the form
of the vector field~\eqref{eqdefvecX}, the submanifold 
$\{\, \vec v \in (\R\times\C)^n \mid v_j = v^*_j \,\}$
is invariant by the maps $\Phi^{t_1,t_2}$
(because $\eta(v_j)=0$ implies that $X_j=0$ on this submanifold),
in particular $\Phi^{t,a}\big((\R\times\C)^n\setminus\cM_\Om^n\big)
\subset (\R\times\C)^n\setminus\cM_\Om^n$,
whence~\eqref{inclusPhiatcMn} follows because~$\Phi^{a,t}$
and~$\Phi^{t,a}$ are mutually inverse bijections.


Therefore the paths 
$\ga_{1}, \ldots, \ga_{n}$
are $\Om$-allowed and have lifts in~$X_\Om$ starting at~$\uO_\Om$, which allow us to define
the maps~$\Psi_t$ by~\eqref{eq:defPsit} on~$\uV$.

We now prove that $(\Psi_t)_{t\in[a,1]}$ is a $\ga$-adapted
deformation of the identity.
The map $(t,\vec v\,) \mapsto \Psi_t(\vec v\,)$ is locally Lipschitz
because the flow map~\eqref{eq:defflowmap} is locally Lipschitz,
and $\Psi_a = \ID$ because $\Phi^{a,a}$ is the identity map of
$(\R\times\C)^n$;
hence, we just need to prove~\eqref{eq:Psitinclus}.

We set
\begin{align*}
%
%
\ti\cN(w) &\defeq \big\{ (v_1,\ldots,v_n) \in (\R\times\C)^n \mid 
v_1 + \cdots + v_n = w \big\}
\ens\text{for $w\in\R\times\C$,}
\\[1ex]
%
%
\ti\cN_j &\defeq \big\{ (v_1,\ldots,v_n) \in (\R\times\C)^n \mid v_j = (0,0) \big\}
\ens\text{for $1\le j\le n$.}
\end{align*}
Let $j\in\{1,\ldots,n\}$. The second part of~\eqref{eq:Psitinclus}
follows from the inclusion
\[
\Phi^{a,t}\big(\ti\cN_j\big) \subset \ti\cN_j
\ens\text{for $t\in[a,1]$},
\]
which stems from the fact that the $j$th component of the vector
field~\eqref{eqdefvecX} vanishes on~$\ti\cN_j$ (because $\eta\big((0,0)\big)=0$).

Since $\ze_1+\cdots+\ze_n = \ga(a) \;\Rightarrow\; 
\abs{\ze_1}+\cdots+\abs{\ze_n} = \abs{\ga(a)}$
for any $(\ze_1,\ldots,\ze_n) \in \uV$,
the first part of~\eqref{eq:Psitinclus} follows from the inclusion
\[
\Phi^{a,t}\big(\ti\cN\big(\ti\ga(a)\big)\big) \subset \ti\cN\big(\ti\ga(t)\big)
\ens\text{for $t\in[a,1]$},
\]
which can be itself checked as follows:
consider first an arbitrary initial condition $\vec v \in
(\R\times\C)^n$ and the corresponding solution $\vec v(t) \defeq
\Phi^{a,t}(\vec v\,)$, and let $v_0(t) \defeq v_1(t) + \cdots +
v_n(t)$; then~\eqref{eqdefvecX} shows that
\[
\frac{\dd\,}{\dd t} \big(\ti\ga(t)-v_0(t)\big)
=\frac{\big| \ti{\ga}(t) -v_0(t) \big|}{D(t,\vec v(t))} \ti\ga'(t), 
\]
hence the Lipschitz function $h(t) \defeq \big| \ti{\ga}(t) -v_0(t) \big|$
has an almost everywhere defined derivative which satisfies
$\abs{h'(t)} \le \big| \frac{\dd\,}{\dd t} \big(\ti\ga(t)-v_0(t)\big)
\big| 
\le \frac{1}{D(t,\vec v(t))} \abs{\ti\ga'(t)} \, h(t)$, 
which is $\le \de\ii\sqrt2\, \abs{\ti\ga'(t)} \, h(t)$ by~\eqref{4.9},
whence
\[
\big| \ti{\ga}(t) -v_0(t) \big| \le 
\big| \ti{\ga}(a) -v_0(a) \big|
\exp\big( \de\ii\sqrt2\, L( \restr{\ga}{[a,t]} ) \big)
\]
for all~$t$;
now, if $\vec v \in \ti\cN\big(\ti\ga(a)\big)$, 
we find $v_0(a) = \ti\ga(a)$, whence $v_0(t) = \ti\ga(t)$ for all~$t$.
\end{proof}


We now show that the $\ga$-adapted deformation of the identity that we
have constructed in Proposition~\ref{prp:constructPsit} meets the
requirements of Theorem~\ref{thm:4.6}.

In view of~\eqref{eq:defcMdeL}--\eqref{eq:defPideL} and~\eqref{eq:defKOmdeL},
the inclusion~\eqref{4.3} follows from
\begin{lmm}    \label{lmm:4.12}
Let $\ti\uV \defeq \big\{
\big( s_1 \ti\ga(a), \ldots, s_n \ti\ga(a) \big) \mid
(s_1,\ldots,s_n) \in \De_n \big\} \in (\R\times\C)^n$.
Then
\beglab{4.13}
\Phi^{a,t}\big( \ti\uV \big) \subset
\bigcup_{L_1+\cdots+L_n=L(\ga_{|t})}
\cM_\Om^{L_1,\de'(t)}
\times\cdots\times
\cM_\Om^{L_n,\de'(t)}
\edla
for all $t\in[a,1]$, with~$\de'(t)$ as in~\eqref{4.5}.
\end{lmm}


\begin{proof}[Proof of Lemma~\ref{lmm:4.12}]
Let us consider an initial condition $\vec v \in \ti\uV$ and the corresponding solution $\vec v(t) \defeq
\Phi^{a,t}(\vec v\,)$,
whose components we write as $v_j(t) = \big( \la_j(t), \ze_j(t) \big)$ 
for $j=1,\ldots,n$.
We also have $v_j(a) = s_j \ti\ga(a)$ for some $(s_1,\ldots,s_n) \in
\De_n$,
whence $\la_1(a) + \cdots + \la_n(a) \le \abs{\ga(a)} = \rho$
and $\abs{v_j(a)} = \rho$ for $j=1,\ldots,n$.

We first notice that 
\[
\sum_{j=1}^n \abs{\la_j'(t)} = 
\sum_{j=1}^n \frac{\eta\big(v_j(t)\big)}{D(t,\vec v(t))} \abs{\ga'(t)}
\le \abs{\ga'(t)},
\]
hence $\la_1(t) + \cdots + \la_n(t) \le 
\la_1(a) + \cdots + \la_n(a) + \int_a^t \abs{\ga'}
\le L(\ga_{|t})$.
Therefore, we just need to show that
\beglab{ineq:needdistSOm}
\dist\big( v_j(t), \cS_\Om \big) \ge \de'(t)
\quad \text{for $j=1,\ldots,n$.}
\edla

Let $j\in\{1,\ldots,n\}$.
   Since~$\eta$ is $1$-Lipschitz, we can define a Lipschitz function
   on $[a,1]$ by the formula
  $h_j(t)\defeq \eta\big(v_j(t)\big)$,
  and its almost everywhere defined derivative satisfies 
\[
\abs{h_j'(t)} \le \abs{v_j'(t)} = \frac{h_j(t)}{D(t,\vec v(t))} \abs{\ti\ga'(t)}
\le g(t) h_j(t),
\quad\text{where}\; g(t) \defeq \de\ii\sqrt2\,\abs{\ga'(t)}.
\]
Since $\int_a^t g(\tau)\,\dd\tau = \de\ii\sqrt2\,L( \restr{\ga}{[a,t]} )$,
we deduce that
\beglab{ineq:etavjt}
\eta\big(v_j(a)\big) \,\ee^{-\de\ii\sqrt2\,L( \restr{\ga}{[a,t]} )}
\le \eta\big(v_j(t)\big) \le
\eta\big(v_j(a)\big) \,\ee^{\de\ii\sqrt2\,L( \restr{\ga}{[a,t]} )}
\quad \text{for all $t\in[a,1]$.}
\edla
Let us now fix $t\in[a,1]$.
We conclude by distinguishing two cases.

Suppose first that 
$\eta(v_j(a)) \ge \rho\, \ee^{-\sqrt{2}\,\de^{-1} L( \restr{\ga}{[a,t]} )}$.
Then the first inequality in~\eqref{ineq:etavjt} yields
$\eta(v_j(t)) \ge \de'(t)$, and since 
$\dist\big( v_j(t), \cS_\Om \big) \ge \eta(v_j(t))$
we get~\eqref{ineq:needdistSOm}.

Suppose now that
$\eta(v_j(a)) < \rho\, \ee^{-\sqrt{2}\,\de^{-1} L( \restr{\ga}{[a,t]} )}$.
Then the second inequality in~\eqref{ineq:etavjt} yields
$\eta(v_j(t')) < \rho$ for all $t'\in[a,t]$.
This implies that $v_j\big([a,t]\big) \subset
B \defeq \{\, v \in \R\times \C \mid \abs{v} < 3\rho/2 \,\}$;
indeed, if not, since $v_j(a)\in B$, there would exist $t' \in (a,t]$ such that
$v_j(t') \in \pa B$,
but using $\Om_{3\rho}=\O$ it is easy to check that 
\[ v\in\ov B \ens\Rightarrow\ens \dist(v,\cS_\Om)\ge 3\rho/2, \]
hence we would have $\dist(v_j(t'),\cS_\Om)\ge
3\rho/2>\eta(v_j(t'))$,
whence $\eta(v_j(t')) = \dist\big( v_j(t'), (0,0) \big) = 3\rho/2$,
which is a contradiction.
Therefore $v_j(t) \in B$, whence $\dist(v_j(t),\cS_\Om)\ge 3\rho/2 > \de'(t)$
and we are done.
\end{proof}


Only the inequality~\eqref{4.4} remains to be proved.
We first show the following:
\begin{lmm}   \label{lmm:auxiLipX}
For any $t\in[a,1]$ and 
$
\vec{u},\vec{v}\in(\R\times\C)^n,
$
the vector field~\eqref{eqdefvecX} satisfies
\begin{equation}\label{4.17}
\sum_{j=1}^{n}
|X_j(t,\vec{u}\,)-X_j(t,\vec{v}\,)|
\leq
3\frac{|\ti\ga'(t)|}{D(t,\vec{u}\,)}
\sum_{j=1}^{n}
|u_j-v_j|.
\end{equation}
\end{lmm}
\begin{proof}[Proof of Lemma~\ref{lmm:auxiLipX}]
We rewrite $X_j(t,\vec{u}\,)-X_j(t,\vec{v}\,)$ as follows:
$$
X_j(t,\vec{u}\,)-X_j(t,\vec{v}\,)
=
\Big(
\eta(u_j)-\eta(v_j)
+\big(D(t,\vec{v}\,)-D(t,\vec{u}\,)\big)
\frac{\eta(v_j)}{D(t,\vec{v}\,)}
\Big)
\frac{\ti\ga'(t)}{D(t,\vec{u}\,)}.
$$
Since
$
|\eta(u_j)-\eta(v_j)|
\leq
|u_j-v_j|
$
holds for $j=1,\ldots,n$,
we have
\begin{align*}
\big|
D(t,\vec{u}\,)-D(t,\vec{v}\,)
\big|
&\leq
\sum_{j=1}^{n}\big|\eta(u_j)-\eta(v_j)\big|
+\bigg|
\Big|\ti\ga(t)
-\sum_{j=1}^{n}u_j\Big|
-\Big|\ti\ga(t)
-\sum_{j=1}^{n}v_j\Big|
\bigg|
\\
&\leq
2\sum_{j=1}^{n}
|u_j-v_j|.
\end{align*}
Then,
summing up 
$
|X_j(t,\vec{u}\,)-X_j(t,\vec{v}\,)|
$
in $j$,
we obtain \eqref{4.17}
from the inequality 
$
\sum_{j=1}^{n}
\eta(v_j)
\leq
D(t,\vec{v}\,).
$
\end{proof}


We conclude by deriving the inequality~\eqref{4.4} from Lemma~\ref{lmm:auxiLipX}.
We use the notation~\eqref{eq:notunzeitzeit} to define 
$\ze_1^t,\ldots,\ze_n^t \col \De_n \to \C$,
and we now define $v_j^t \col \De_n\to\R\times\C$ for $t\in[a,1]$ by
the formulas $v_j^a(\vec s\,)
\defeq s_j \ti\ga(a)$ and
\[
\vec v\,^t \defeq (v_1^t,\ldots,v_n^t) 
\defeq \Phi^{a,t} \circ (v_1^a,\ldots,v_n^a).
\]
Let
\[
V(t)
\defeq 
\sum_{j=1}^{n}
|\ze_j^t(\vec{s}\,)-\ze_j^t(\vec{s}\,')|
\quad\text{for $\vec{s},\vec{s}\,'\in\De_n$.}
\]
We obtain from 
\eqref{4.9} and
\eqref{4.17} 
the following estimate:
\begin{align*}
V(t)
&\leq
V(a)+
\frac{1}{\sqrt{2}}
\sum_{j=1}^{n}
\int_a^t
\big|X_j\big(\tau,\vec v\,^{\tau}(\vec{s}\,)\big)-
X_j\big(\tau,\vec v\,^{\tau}(\vec s\,')\big)\big| \,\dd \tau
\\
&\leq
V(a)+
\frac{3}{\de}\int_a^t |\ga'(\tau)|V(\tau)d\tau.
\end{align*}
Therefore,
Gronwall's lemma yields
$
V(t)
\leq
V(a)\,\ee^{3\de\ii L( \restr{\ga}{[a,t]} )},
$
and hence,
since 
$
V(a)=\rho
\sum_{j=1}^{n}|s_j-s_j'|,
$
we have
\begin{equation}\label{4.18}
V(t)
\leq
\rho\,\ee^{3\de\ii L( \restr{\ga}{[a,t]} )}
\sum_{j=1}^{n}|s_j-s_j'|.
\end{equation}
Then,
\eqref{4.18} entails
via Rademacher's theorem
that
the following estimate holds a.e.\ on $\De_n$:
$$
\sum_{i=1}^{n}
\bigg|
\frac{\partial \zeta^t_i}{\partial s_j}
\bigg|
\leq
\rho\,\ee^{3\de\ii L( \restr{\ga}{[a,t]} )}.
$$
Finally,~\eqref{4.4} follows from the inequality
$$
\bigg|\Det\Big[\, \frac{\pa\ze_i^t}{\pa s_j}\,\Big]_{1\le i,j\le n}\bigg|
\leq
\prod_{j=1}^{n}
\bigg(
\sum_{i=1}^{n}
\bigg|
\frac{\partial \zeta^t_i}{\partial s_j}
\bigg|
\bigg).
$$

\begin{rem}
Theorem \ref{thm:4.9} is verified by replacing the vector field~\eqref{eqdefvecX} by
$$
\vec X(t,\vec v\,) =  \left| \begin{aligned}
X_1 &\defeq \frac{\eta_1(v_1)}{D(t,\vec v\,)}\ti\ga'(t)
\\
& \qquad \vdots \\[1ex]
X_n &\defeq \frac{\eta_n(v_n)}{D(t,\vec v\,)}\ti\ga'(t),
\end{aligned} \right.
$$
where
$ 
\eta_j(v) \defeq \dist\big( v, \{ (0,0) \} \cup \ov\cS_{\Om_j} \big) 
$, 
$ 
D\big(t,\vec{v}\,\big) \defeq
\eta_1(v_1) + \cdots + \eta_n(v_n) + 
|\ti\ga(t)-(v_1+\cdots+v_n )|.
$ 
\end{rem}

\subsection{The case of endless continuability \bdv}\label{sec:multifilt}

In this subsection,
we extend the estimates of
Theorem~\ref{thm:4.7} to the case of a \ddfs.

\begin{nota}
Given $\de,M,L>0$,
we denote by $\Pi_\Om^{\de,M,L}$
the set of all paths $\ga\in\Pi_\Om\dv$
such that 
$V(\ga)\leq M$, $L(\ga)\leq L$ and 
$
\inf\limits_{t\in[0,t_*]}
\dist_1\big(
\ti\ga\dv(t),\ov\cS_{\Om}\big)\geq\de
$,
where $\ti\ga\dv$ is as in~\eqref{eqdeftigadv} and
$\dist_1$ is 
the distance associated with the norm $\|\cdot\|_1$
defined on $\R^2\times\C$
by $\|(\mu,\la,\ze)\|_1\defeq |\mu|+\sqrt{|\la|^2+|\ze|^2}$.
\end{nota}

Let us fix an arbitrary \ddfs~$\Om$.
We fix $\rho>0$ such that $\Om_{3\rho,M}=\O$
for every $M\geq0$.
We consider a path $\ga:[0,1]\to\C$ in $\Pi_{\Om^{*n}}^{\de,M,L}$,
with arbitrary $\de\in(0,\rho)$ and $L>0$,
satisfying the following condition:
\begin{eq-text}
There exists $a\in(0,1)$ such that
$\ga(t)=\frac{t}{a}\ga(a)$ for $t\in[0,a]$ and $|\ga(a)|=\rho$.
\end{eq-text}
Then, for $t\in[0,1]$ and $v \in \R\times\C$, we set
\[
\eta(t,v) \defeq 
\dist_1\big( (V(\ga_{|t}),v), \big(\R\times \{ (0,0) \} \big) \cup \ov\cS_{\Om} \big).
\]
and, for $\vec{v}=(v_1,\cdots,v_n) \in (\R\times\C)^n$,
\[
D\big(t,\vec{v}\,\big) \defeq
\eta(t,v_1) + \cdots + \eta(t,v_n) + 
%
|\ti\ga(t)-( v_1+\cdots+  v_n )|.
\]
%
%
Choosing
$(\mu_j,u_j)\in \big( \R\times \{ (0,0) \} \big)\cup\ov\cS_{\Om}$
so that
$\eta(t,v_j)=\|(V(\ga_{|t}),v_j)-(\mu_j,u_j)\|_1$
for each~$j$
and using
$
(\mu_{j_0}, u_1+\cdots+  u_n)
\in\big( \R\times \{ (0,0) \}  \big)\cup\ov\cS_{\Om^{*n}}
$
with
$\dst\max_{j=1,\cdots,n}\mu_j=\mu_{j_0}$,
we see that
\begin{align*}
D(t,\vec{v}\,)
&=
\sum_{j=1}^n
(|V(\ga_{|t})
-\mu_j|+| v_j- u_j|)
+|\ti\ga(t)-( v_1+\cdots+  v_n )|
\\
&\geq
\big|V(\ga_{|t})-
\mu_{j_0}\big|+
|\ti\ga(t)-(  u_1+\cdots+  u_n )|
\\
&\geq\min\{\de,\rho\}
\end{align*}
for $(t,\vec v\,) \in [a,1]\times (\R\times\C)^n $.

We can thus define a map
$
(t^*,t,\vec v\,) \in [a,1]\times[a,1]\times (\R\times\C)^n 
\mapsto \Phi^{t^*,t}(\vec v\,) \in (\R\times\C)^n
$
as the flow map of
\begin{equation}\label{4.32}
\vec X(t,\vec v\,) =  \left| \begin{aligned}
X_1(t,\vec v\,) &\defeq 
\frac{\eta(t,v_1)}{D(t,\vec v\,)}\ti\ga'(t)
\\
& \qquad \vdots \\[1ex]
X_n(t,\vec v\,) &\defeq
\frac{\eta(t,v_n)}{D(t,\vec v\,)}\ti\ga'(t)
\end{aligned} \right.
\end{equation}
Let $\vec{v}\,^t=(\vec{v}\,^t_1,\cdots,\vec{v}\,^t_n)$ be the flow of \eqref{4.32}
with the initial condition 
$
\vec{v}\,^a_j\defeq(|\ga(a)|s_j,\ga(a)s_j)
$
with $\vec{s}\in\De_n$.
Since $\ti\ga'(t)$,
$\eta(t,v_j)$ and $D(t,\vec v\,)$ are Lipschitz continuous on
$[a,1]\times (\R\times\C)^n $,
we find by Rademacher's theorem that
$d\ze_j^t/dt$ is differentiable a.e.\ on $[a,1]$
and satisfies
\begin{equation*}
\frac{d^2\ze_j^t/dt^2}{d\ze_j^t/dt}
=\frac{1}{\eta(v_j^t)}\frac{d\eta(v_j^t)}{dt}
-\frac{1}{D(t,\vec{v}\,^t)}\frac{dD(t,\vec{v}\,^t)}{dt}
+\frac{\ga''(t)}{\ga'(t)}
\end{equation*}
when $s_j\neq0$.
Since $\eta(v_j^t)$ and $D(t,\vec{v}\,^t)$ are real valued functions,
we have
$$
{\rm Im}\frac{d^2\ze_j^t/dt^2}{d\ze_j^t/dt}
={\rm Im}\frac{\ga''(t)}{\ga'(t)}.
$$
Therefore, the following holds
for every $t\in [a,1]$:
\begin{equation*}
V\big(\ze_j^\cdot|_{[0,t]}\big)
=V(\ga_{|t})
\quad
\text{when}
\quad
s_j\neq0.
\end{equation*}
Arguing as for Theorem \ref{thm:4.7}, we obtain
\begin{thm}
Let $\de,L,M>0$ be reals.
Then there exist $c,\de'>0$ such that,
for every \ddfs~$\Om$ such that $\Om_{4\de,M}=\O$ $(M\geq0)$,
for every integer $n\ge1$ and for
every
$\hat f_1, \ldots, \hat f_n \in \hat\gR_\Om\dv$,
the function $1*\hat f_{1}*\cdots *\hat f_{n}$ belongs to $\hat\gR_{\Om^{*n}}\dv$
and satisfies
\beglab{4.33}
 \big\|\,
\fp_{\Om^{*n}}^* \big( 1*\hat f_{1}*\cdots *\hat f_{n}\big )
\big\|_{\Om^n}^{\de,M,L}
\leq
\frac{c ^n}{n!}
%
\ \sup_{L_1+\cdots+L_n=L}
 \big\|\, \fp_\Om^{*} \hat f_1 \big\|_{\Om}^{\de',M,L_1}
\, \cdots
\big\|\, \fp_\Om^{*} \hat f_n \big\|_{\Om}^{\de',M,L_n},
\edla
where the seminorm 
$\|\cdot\|_{\Om^n}^{\de,M,L}$
on $\hat\gR_\Om\dv$ is defined by the supremum on the set
$
\big\{
\un\ga(1)\mid
\ga\in\Pi_\Om^{\de,M,L}
\big\}.
$
\end{thm}



\section{Applications}\label{sec:5}


In this section,
we display some applications of
our results of Section \ref{sec:4}.
We first introduce convergent power series
with coefficients in $\ti\gR_\Om$:

\begin{dfn}
Given $\Om$ a \dfs\ and $r\ge1$,
we define $\ti\gR_\Om\{w_1,\cdots,w_r\}$ 
as the space 
of all 
$$
\ti F(z,w_1,\cdots,w_r)
=\sum_{k\in\Z_{\geq0}^{\,r}}
\ti F_{k}(z)w_1^{k_1}\cdots w_r^{k_r}
\in \ti\gR_\Om[[w_1,\cdots,w_r]]
$$
such that,
for every $\de,L>0$,
there exists a positive constant $C$
satisfying
$$
\|\, \ti F_{k}
\,
\|_{\Om}^{\de,L}
\leq C^{|k|+1}
\quad\text{for every $k=(k_1,\cdots,k_r)\in\Z_{\geq0}^{\,r}$,}
$$
where $|k| \defeq k_1+\cdots+k_r$
(with the notation of Definition~\ref{dfn:3.8} for $\| \cdot \|_{\Om}^{\de,L}$).
\end{dfn}

We can now deal with the substitution 
of resurgent formal series in a context more general than in Theorem~\ref{thmsubstgR}.

\begin{thm}   \label{thmsubstOmgR}
Let $r\ge1$ be integer and let $\Om_0$, \ldots, $\Om_r$ be \dfs\ 
Then for any 
$\ti F(w_1,\ldots,w_r) \in
\ti\gR_{\Om_0}\{w_1,\cdots,w_r\}$
and for any $\ti\ph_1,\ldots,\ti\ph_r \in \C[[z\ii]]$ without constant
term, one has 
\[
\ti\ph_1 \in \ti\gR_{\Om_1}, \ldots, \ti\ph_r \in \ti\gR_{\Om_r}
\quad \Rightarrow \quad
\ti F(\ti\ph_1,\ldots,\ti\ph_r) \in \ti\gR_{\Om_0*\Om^{*\infty}},
\]
where $\Om \defeq \Om_1 * \cdots * \Om_r$.
\end{thm}


\begin{proof}
Since the family 
$
\{\Om_0*\Om^{*k}\defeq \Om_0*\Om_1^{*k_1}*\cdots *\Om_r^{*k_r}
\ |\ 
k=(k_1,\cdots,k_r)\in\Z_{\geq0}^{\,r}
\}
$
of \dfs\ satisfies the conditions in Theorem \ref{thm:4.9}
for sufficiently small $\de>0$,
for every $L>0$,
there exist $\de',L',C>0$ such that
$$
\big\|\, 
\ti F_k\ti\ph^{k_1}\cdots\ti\ph^{k_r}
\big\|_{\Om_0*\Om^{*\infty}}^{\de,L}
\leq
\frac{C^{|k|+2}}{(|k|+1)!}
\big\|\, 
\ti F_k
\big\|_{\Om_0}^{\de',L'}
\big(\big\|\, 
\ti\ph_1
\big\|_{\Om_0}^{\de',L'}\big)^{k_1}
\cdots
\big(\big\|\, 
\ti\ph_r
\big\|_{\Om_0}^{\de',L'}\big)^{k_r}.
$$
Therefore,
since $\ti F(w_1,\ldots,w_r) \in
\ti\gR_{\Om_0}\{w_1,\cdots,w_r\}$,
we find that
$\ti F(\ti\ph_1,\ldots,\ti\ph_r)$ converges in 
$\ti\gR_{\Om_0*\Om^{*\infty}}$
and defines an $\Om_0*\Om^{*\infty}$-resurgent formal series.
\end{proof}

Notice that,
in view of Theorem~\ref{propidgROmgR}, Theorem~\ref{thmsubstgR}
is a direct consequence of Theorem \ref{thmsubstOmgR}.

Next,
we show the following implicit function theorem
for resurgent formal series:
\begin{thm}   \label{thm:IRFT}
Let $\ti F(z,w)\in \ti\gR_\Om\{w\}$ and
assume that $F(x,w) \defeq \ti F(x\ii,w)$ satisfies
$F(0,0)=0$
and $\pa_wF(0,0)\neq0$.
Then, the unique solution $\ti\ph\in z\C[[z]]$ of
\begin{equation}\label{5.1}
\ti F(z,\ti\ph(z))=0
\end{equation}
satisfies
$\ti\ph\in\ti\gR_{\Om^{*\infty}}$.
\end{thm}

\begin{proof}
We rewrite $\ti F(z,w)$ into the form
$$
\ti F(z,w)
=\ti F_0(z)+\pa_wF(0,0)w+
\sum_{k=1}^\infty
\ti F_k(z)w^k.
$$
Considering \eqref{5.1}
as the equation for
$
\ti\psi=z^{-1}(
\ti\ph(z)-\ph_1z),
$
we can assume that~$\ti F_k$ has no constant term for $k=0,1,\ldots$
Further,
we can assume without loss of generality
that $\pa_wF(0,0)=-1$.
Then, the unique solution
$\ti\ph\in z\C[[z]]$ of \eqref{5.1}
can be written as $\ti\ph=\ti H(z,\ti F_0)$,
where
$$
\ti H(z,w)=
\sum_{m\geq1}
\ti H_m(z)w^m,
\qquad
\ti H_m\defeq 
\sum_{k\geq1}
\frac{(m+k-1)!}{m!k!}
\sum_{
\substack{
n_1+\cdots+n_k=m+k-1
\\
n_1,\cdots,n_k\geq1
}
}
\ti F_{n_1}\cdots \ti F_{n_k}
$$
(see proof of Theorem 4 in \cite{NLresur}
for the detail).
Since $\ti F(z,w)\in\ti\gR_\Om\{w\}$,
we obtain from Corollary~\ref{crl:4.7'}
the following estimates:
For every $\de,L>0$,
there exist $\de',L',C>0$ such that
$
\|\, F_k\,
\|_{\Om}^{\de',L'}
\leq C^{k+1}
$
and
\begin{align*}
\|
\ti H_{m}
\|_{\Om^{*\infty}}^{\de,L}
&\leq
\sum_{k\geq1}
\frac{(m+k-1)!}{m!k!}
\sum_{
\substack{
n_1+\cdots+n_k=m+k-1
\\
n_1,\cdots,n_k\geq1
}
}
\frac{C^{k+1}}{k!}
\|\ti F_{n_1}\|_{\Om}^{\de',L'}
\cdots 
\|\ti F_{n_k}\|_{\Om}^{\de',L'}
\\
&\leq
\sum_{k\geq1}
2^{m+k}
\sum_{
\substack{
n_1+\cdots+n_k=m+k-1
\\
n_1,\cdots,n_k\geq1
}
}
\frac{C^{m+3k}}{k!}
\\
&\leq
\sum_{k\geq1}
2^{2m+3k-2}
\frac{C^{m+3k}}{k!}
\\
&\leq
e^{8C^3}
(4C)^{m}.
\end{align*}
This yields
$\ti H(z,w)\in\ti\gR_{\Om^{*\infty}}\{w\}$,
whence,
$\ti H(z,\ti F_0(z))\in\ti\gR_{\Om^{*\infty}}.$
\end{proof}

\medskip

\noindent {\em Acknowledgements.}
{This work has been supported by Grant-in-Aid for JSPS Fellows Grant
  Number 15J06019, French National Research Agency reference
  ANR-12-BS01-0017 and Laboratoire Hypathie A*Midex.
The authors thank Fibonacci Laboratory (CNRS UMI~3483), the Centro Di
Ricerca Matematica Ennio De Giorgi and the Scuola Normale Superiore di
Pisa for their kind hospitality.
}


\ifx\undefined\bysame 
\newcommand{\bysame}{\leavevmode\hbox to3em{\hrulefill}\,} 
\fi 












\end{document}